\documentclass[11pt,leqno,english]{amsart}
\makeatletter
\def\paragraph{\@startsection{paragraph}{4}%
  \z@{.5\linespacing}{.5\linespacing}%
  {\bfseries}}
\def\@secnumfont{\ifnum\@toclevel>1\bfseries\else\mdseries\fi}
\makeatother

\usepackage[left=1.3in,top=1.3in,right=1.3in,nofoot]{geometry}
\usepackage{amsmath}
\usepackage{amsfonts}
\usepackage{amssymb}
\usepackage{amscd}
\usepackage[titletoc,title]{appendix}
\usepackage{appendix}
\usepackage{hyperref}

\newcommand{\Z}{\mathbb{Z}}

\newcommand{\cn}{\mathcal{N}}

\newcommand{\cx}{\mathcal{X}}
\newcommand{\cz}{\mathcal{Z}}
\newcommand{\cV}{\mathcal{V}}
\newcommand{\cy}{\mathcal{Y}}

\newcommand{\bb}{\mathbf{B}}

\newcommand{\bg}{\mathbf{G}}

\newcommand{\bm}{\mathbf{M}}

\newcommand{\bh}{\mathbf{H}}
\newcommand{\bn}{\mathbf{N}}
\newcommand{\bp}{\mathbf{P}}

\newcommand{\bu}{\mathbf{U}}
\newcommand{\bz}{\mathbf{Z}}

\DeclareMathOperator{\id}{{id}}
\DeclareMathOperator{\St}{{St}}

\DeclareMathOperator{\GL}{GL}

\DeclareMathOperator{\Hom}{{Hom}}
\DeclareMathOperator{\End}{{End}}

\DeclareMathOperator{\Ind}{{Ind}}
\DeclareMathOperator{\ind}{{ind}}

\DeclareMathOperator{\Mod}{{Mod}}

\DeclareMathOperator{\Aut}{{Aut}}

\DeclareMathOperator{\Id}{{Id}}

\newtheorem{montheo}{Theorem}
\newtheorem{monlem}{Lemma}
\newtheorem{maprop}{Proposition}
\newtheorem{madef}{Definition}
\newtheorem{moncor}{Corollary}
\newtheorem{remark}{Remark}

\usepackage{xcolor}
\definecolor{teal}{rgb}{0.0, 0.5, 0.5}
\definecolor{forest}{rgb}{0.13, 0.55, 0.13}

\begin{document}

\title[Representations of a  $p$-adic group in characteristic $p$]
{Representations of a   $p$-adic group  in characteristic $p$}


\author{G. Henniart} 
\address[G. Henniart]{Universit\'e de Paris-Sud, Laboratoire de Math\'ematiques d'Orsay, Orsay cedex F-91405 France;
CNRS, Orsay cedex F-91405 France}
\email{Guy.Henniart@math.u-psud.fr}


\author{M.-F. Vign\'eras}
\address[M.-F. Vign\'eras]{Institut de Math\'ematiques de Jussieu-Paris Rive Gauche, 4 Place Jussieu, Paris 75005 France}
\email{marie-france.vigneras@imj-prg.fr}

\date{}

\begin{abstract}
Let $F$ be a locally compact non-archimedean field of residue characteristic $p$,   $\textbf{G}$ a connected reductive group over $F$, and $R$ a field of characteristic $p$. When $R$ is  algebraically closed,  the   irreducible admissible $R$-representations of $G=\bg(F)$  are classified in \cite{AHHV} in term of supersingular $R$-representations of the Levi subgroups of $G$ and parabolic induction; there is a similar classification  for the simple modules of the pro-$p$ Iwahori Hecke $R$-algebra $H(G)_R$ in \cite{Abe}.  
In this paper, we show that both classifications   hold true when $R$ is not  algebraically closed.    \end{abstract}
\maketitle






\tableofcontents

\section{Introduction}\label{I}

\subsection{}\label{I.1} 
In this paper, $p$ is a prime number, $F$ is a locally compact non-archimedean field of residual characteristic $p$,  $\bg$ is a connected reductive group over $F$, finally $R$ is a field; in this introduction $R$ has   characteristic $p$ - except in \S \ref{I.2}.

Recent applications of automorphic forms to number theory have imposed the study of smooth representations of $G=\bg(F)$  on $R$-vector spaces; indeed one expects a strong relation, \`a la Langlands, with $R$-representations of the Galois group of $F$ - the only established case, however, is that of $GL(2,\mathbb Q_p)$.

The first focus is on irreducible representations.
 When  $R$ is an algebraically closed, the irreducible admissible   $R$-representations  of  $G$ have been classified in terms of parabolic induction of supersingular  $R$-representations of Levi subgroups of $G$  \cite{AHHV}. But the restriction to algebraically closed $R$ is undesirable: for example, in the work of Breuil and Colmez on $GL(2,\mathbb Q_p)$, $R$ is often finite. Here we extend to any $R$ the classification of \cite{AHHV} and its consequences.
 
 \bigskip 
 Let $I$ be a pro-$p$ Iwahori subgroup of $G$. If $W$ is a smooth $R$-representation of $G$, the  fixed point $W^I$ is a right module over the Hecke ring $H(G)$ of $I$ in $G$; it is non-zero if $W$ is, and finite dimensional if $W$ is admissible. Even though $W^I$ might not be simple over $H(G)$ when $W$ is irreducible, it is important to study simple $R\otimes H(G)$-modules. When $R$ is 
 algebraically closed, they have been 
  classified  (\cite{Abe}, see also \cite[Cor:4.30]{AHenV2}) in terms of supersingular 
  $R\otimes H(M)$-modules, where $M$ is a Levi subgroup of $G$ and $H(M)$ the 
  Hecke ring  of $I\cap M$ in $M$. The classification uses  a parabolic induction process from $H(M)$-modules  to $H(G)$-modules. Again we extend that classification to any $R$.
   
\subsection{} \label{I.2}Before we state our main results more precisely, let us describe our principal tool for reducing them to the known case where $R$ is algebraically closed - those tools are developped in section \ref{S:2}.

The idea is to introduce an algebraic closure $R^{alg}$ of $R$, and study the scalar extension $W\mapsto R^{alg}\otimes_R W$ from $R$-representations of $G$ to $R^{alg}$-representations of $G$, or from $R\otimes H(G)$-modules to $R^{alg}\otimes H(G)$-modules. The important remark is that when $W$ is an irreducible admissible $R$-representation of $G$, or a simple  $R\otimes H(G)$-module, its commutant has finite dimension over $R$. The following result examines what happens for more general extensions $R'$ of $R$.

\begin{montheo}\label{thm:DA}{\rm[Decomposition theorem]}  Let $R$ be a field, $A$ an  $R$-algebra\footnote{all our algebras are associative with unit}  and $V$ a simple $A$-module with  commutant  $D=\End_A V$  of finite dimension over $R$.  
     Let $E$ denote the center of  the skew field $D$, $\delta$ the reduced degree of $D$ over $E$ and $E_{sep}$ the maximal separable extension of $R$ contained in $E$.       

Let   $R'$  be a  normal extension of $R$ containing  a   finite separable  extension   of $E$ splitting $D $. Then 
the scalar extension $V_{R'}$ of $V$ to $R'$ has length $\delta [E:R]$ and is a direct sum 
  $$V_{R'} \simeq \oplus^\delta \oplus_{j\in \Hom_R(E_{sep},R')}V_{R',j}$$
of $\delta$ copies of a direct sum of  $[E_{sep}:R]$ indecomposable  $A_{R'}$-modules $V_{R',j}$ 
of commutant the local artinian ring $R'\otimes_{j, E_{sep}} E$. For each $j$, $V_{R',j}$ has  length  $[E:E_{sep}]$,  simple subquotients all  isomorphic to 
   $ R' \otimes _{( R'\otimes_{ E_{sep}}E)}V_{R',j} $ of  commutant $R'$, and descends to a finite extension of $R$. The $V_{R',j}$ are not isomorphic to each other and    form   a single $ \Aut_R(R')$-orbit.   
   
      The  map sending $V$ to  the $ \Aut_R(R')$-orbit of $ R' \otimes _{( R'\otimes_{ E_{sep}}E)}V_{R',j}  $ 
   induces a bijection 

-  from   the set of isomorphism classes $[V]$ of simple $A$-modules $V$ with  commutant  of  finite dimension over $R$ (resp.  $V$ of finite dimension over $R$),

- to  the set of  $\Aut_R (R')$-orbits  of the isomorphism classes  $[V']$ of   simple $A_{R'}$-modules $V'$ with  commutant of finite dimension over $R'$   descending to a finite extension of $R$ (resp. $V'$ of finite dimension over $R'$).
 \end{montheo}
 
  Thm.\ref{thm:DA} implies without difficulty:

\begin{moncor}\label{cor:DA}  
 For any  extension $L/R$,
the length of $V_{L} $ is $\leq \delta [E:R]$,
 and  the dimension over $L$ of the commutant of any subquotient of $V_{L}$  is finite.\end{moncor} 

  The second theorem is a criterion, inspired by \cite[Lemma 3.11]{AHenV1},  for  a  functor  to  preserve   the  lattice $\mathcal L_W$ of submodules of a module $W$.   
  
 \begin{montheo}{\rm[Lattice isomorphism]}\label{thm:latticeA} Let $F : \mathcal C\to  \mathcal D$ be  a functor   between abelian categories  of right adjoint $G$, unit $\eta:\id\to G\circ F$  and counit $\epsilon:F\circ G\to \id$.   
 
 Then 
   $F$ and $G$  induce  lattice isomorphisms inverse from each other  between   $\mathcal L_W$   and  $\mathcal L_{ F(W)}$, 
for any object  $W\in  \mathcal C $  of finite length satisfying the properties: 
 
 a)  The  unit morphism $\eta_W: W\to (G\circ F)(W)$ is an isomorphism.
 
b) $F(Y_1)\subset F(Y_2)$ for any  subobjects $Y_1\subset Y_2$ of $W$,  or $G(X)\neq 0$ for any non-zero  subobject  $X$ of $F(W)$.

c)   $F(Y)$ is simple for any  simple object $Y$ subquotient of $W$.

\noindent equivalent to the properties a), b'), c') where 
 
 b')  $G(X)$ is simple for any  simple   subquotient $X$ of $F(W)$.
 
c') The length of $F(W)$ is equal to the length of $W$.
 
\medskip  If $W$ satisfies a), b), c), then  any subquotient of $W$   satisfies  a), b) and c).

\end{montheo}
The proof starts by showing that $F$ defines an injective lattice map  $\mathcal L_W \hookrightarrow\mathcal L_{ F(W)}$ using a) and b), and then showing that the lattice map $\mathcal L_{F(W)} \hookrightarrow\mathcal L_{ (G\circ F)(W)}$ defined by $G$ is injective using c).

We end \S\ref{S:2} by another lattice isomorphism  for    the  functor $- \otimes_R V$ from $R$-vector spaces to $A$-modules, where $V$ is  a  simple $A$-module with commutant $R$,   inspired by \cite[Lemma 5.3]{Abe}.

\begin{montheo}{\rm[Lattice isomorphism and tensor product]} \label{thm:latticeA2} 
Let $R$ be a field, $A$ an $R$-algebra, $V$ a simple $A$-module with commutant $R$, and $W$ an $R$-vector space. Then, 

$W\otimes_R V$ is an isotypical $A$-module of type $V$,

the map $Y\mapsto Y\otimes_R V:\mathcal L_W\to \mathcal L_{W\otimes_R V}$  is a  lattice isomorphism,

the map $ W\to  \Hom_A(V,W\otimes_R V)$, sending $w\in W$ to $\varphi_w:v\mapsto w\otimes v$ is an isomorphism.

\noindent Let moreover, $b_W \in \End_R(W)$,  $b_V \in \End_R(V)$ and a subspace $Y$ of $W$.  
If $Y$ is stable by $b_W$, then $Y\otimes_R V$ is stable by $b_W \otimes b_V$. Conversely, if $Y\otimes_R V$ is stable by $b_W \otimes b_V$, then $Y$ is stable by $b_W$ provided that $b_V\neq 0$.
\end{montheo}
In our applications, we have an $R$-algebra $A'$ containing $A$ and an $A'$-module $V$ which is simple with commutant $R$ as an $A$-module. We also have a basis $B'$ of $A'$ containing a basis $B$ of $A$ and elements of $B'\setminus B$ act invertibly on $V$. Moreover, we deal with $A'$-modules $W$ where elements of $B$ act as identity, and such that the tensor product action of $B'$ on $W\otimes_R V$ yields an $A'$-module.

\begin{moncor}\label{cor:latticeA2} In the above situation,

The map $Y\mapsto Y\otimes_R V:\mathcal L_W\to \mathcal L_{W\otimes_R V}$ yields a lattice isomorphisms between $A'$-submodules of $W$  and $A'$-submodules
of $W\otimes_R V$.

The map $ W\to  \Hom_A(V,W\otimes_R V)$ is an isomorphism of $A'$-modules, if we let $b\in B'$ act on $\varphi\in \Hom_A(V, W \otimes_R V)$ by $b\varphi= b_{W \otimes_R V}\circ \varphi \circ b_V^{-1}$.
\end{moncor}
Note that the natural map  $\Hom_A(V, W \otimes_R V)\otimes_RV \to W\otimes_R V$ is also an isomorphism of $A'$-modules if we let $b\in B'$ act by the tensor product action.

 \subsection{}\label{s:13} In  \S \ref{S:3}, for a field $R$ of characteristic $p$, we prove the classification of  the irreducible admissible $R$-representations of $G$ in terms of supersingular irreducible admissible $R$-representations of Levi subgroups of $G$.

An $R$-triple $(P,\sigma,Q)$ of $G$  consists of a parabolic subgroup $P=MN$ of $G$, a smooth  $R$-representation $\sigma$ of $M$, and  a parabolic subgroup  $Q $ of $G$ satisfying $P\subset Q \subset P(\sigma)$, where $P(\sigma)=M(\sigma)N(\sigma) $ is the maximal parabolic subgroup of $G$ where  $\sigma$,   extends trivially on $N$; let  $e_Q(\sigma)$ denote the restriction to $M_Q$  of this extension. By definition   
\begin{align}\label{eq:IG}I_G(P,\sigma,Q)&=\Ind_{P(\sigma)}^G (\St_Q^{M(\sigma)}(\sigma)) \quad \text{where}\\
 \label{eq:IGst}\St_Q^{M(\sigma)}(\sigma)&= \Ind_Q^{M(\sigma)}(e_Q(\sigma) )/\sum_{Q\subsetneq Q'\subset P(\sigma)}\Ind_{Q'}^{M(\sigma)} (e_{Q'}(\sigma)) , 
\end{align}
is the  generalized Steinberg $R$-representation   of $ M(\sigma) $ and $\Ind_Q^{M(\sigma)}$ stands for  the parabolic smooth induction $\Ind_{Q\cap M(\sigma)}^{M(\sigma)}$.  In \S\ref{S:II.1} Prop.\ref{prop:es2}, we show that  $I_G(P,-,Q)$  and scalar extension are compatible: for any $R$-triple $(P,\sigma,Q)$ of $G$, we have $R'\otimes_RI_G(P,\sigma,Q)\simeq I_G(P,R'\otimes_R\sigma,Q)$ for any extension $R'/R$ and $I_G(P,\sigma,Q)$ descends to a subfield of $R$ if and only if $\sigma$ does,

What  supersingular means for  an irreducible  smooth $R$-representation $\pi$ of $G$ ?   We know what it means to be a  supersingular $H(G)_R=R\otimes_{\mathbb Z} H(G)$-module:  for all $P\neq G$, a
  certain central element $T_P$ of  the pro-$p$ Iwahori Hecke ring $H(G)$ should  act locally nilpotently \cite{VigpIwss}.  
   We say that   $\pi$  is   {\bf supersingular} if  $\pi^I$ (the $I$-invariants)  is  supersingular as  a  right $H(G)_R$-module (Definition \ref{def:ss}  in \S\ref{S:ss}). 
 When $R$ is algebraically closed, the  definition given in  \cite{AHHV}  is  equivalent  by \cite{OV}. In  \S\ref{S:ss} Lemma \ref{lem:ss}, we show that   supersingularity is compatible with scalar extension.

\begin{montheo}\label{thm:classG} {\rm [Classification theorem for $G$]}

For any $R$-triple $(P,\sigma,Q)$ of $G$ with $\sigma$ irreducible admissible supersingular, $I_G(P,\sigma,Q)$ is an irreducible admissible $R$-representation  of $G$.

If $(P,\sigma,Q)$ and $(P_1,\sigma_1,Q_1)$ are   two $R$-triples of $G$ with $\sigma$ and 
$\sigma_1$ irreducible admissible supersingular and  $I_G(P,\sigma,Q)\simeq I_G(P_1,\sigma_1,Q_1)$, then $P=P_1,Q=Q_1$ and $\sigma\simeq \sigma_1$. 

For any irreducible admissible $R$-representation $\pi$  of $G$,  there is a $R$-triple $(P,\sigma,Q)$ of $G$ with $\sigma$ irreducible admissible supersingular,  such that $\pi\simeq I_G(P,\sigma,Q)$. 
\end{montheo}

When  $R$ is algebraically closed, this is the   classification theorem  of \cite{AHHV}. In   \S\ref{S:3.4} we  descend the classification theorem   from $R^{alg}$ to $R$ by a formal proof using  the decomposition theorem (Thm.\ref{thm:DA}) and a lattice isomorphism $\mathcal L_{\sigma_{R^{alg}}}\simeq \mathcal L_{I_G(P,\sigma_{R^{alg}},Q)}$ when $\sigma$ is irreducible admissible supersingular of scalar extension $\sigma_{R^{alg}}$ to  $R^{alg}$ (Prop.\ref{prop:tensorstG} in \S\ref{S:II.1}, Remark \ref{rem:tensorstG} in \S\ref{S:ss}).

\subsection{}\label{I.6}  
In  \S \ref{S:4}, we prove a similar classification for   the simple right $H(G)_R$-modules when $R$ is a field of characteristic $p$. As in \cite{AHenV2} when $R$ is algebraically closed, this classification uses the parabolic induction functor 
$$\Ind_P^{H(G)}:\Mod_R(H(M))\to \Mod_R(H(G))$$ from right $H(M)_R$-modules to right $H(G)_R$-modules, analogue  of the parabolic smooth induction:  indeed $(\Ind_P^G \sigma)^I $ is naturally isomorphic to $  \Ind_{P}^{H(G)}(\sigma^{I\cap M})$ for a smooth $R$-representation $\sigma$ of $G$  \cite{OV}. An $R$-triple  $(P,\cV,Q)$ of $H(G)$ consists of parabolic subgroups $P=MN\subset Q$ of $G$ (containing $B$) and of a right $H(M)_R$-module $\cV$ with $Q\subset P(\cV)$ (Definition \ref{def:triH}); as for the group, it defines a right $H(G)_R$-module $I_{H(G)}(P,\cV,Q)$.

\begin{montheo}\label{thm:classHG} {\rm [Classification theorem for $H(G)$]}

Any simple right $H(G)_R$-module $\cx$ is isomorphic to $I_{H(G)}(P,\cV,Q)$ for some $R$-triple $(P,\cV,Q)$ of $H(G)$ with  $\cV$ simple supersingular.

For any $R$-triple $(P,\cV,Q)$ of  $H(G)$  with $\cV$ simple supersingular, $I_{H(G)}(P,\cV,Q)$ is  a  simple $H(G)_R$-module.

If  $(P,\cV,Q)$ and $(P_1,\cV_1,Q_1)$ are  $R$-triples of $H(G)$ with $\cV$ and 
$\cV_1$ simple supersingular,  and $I_{H(G)}(P,\cV,Q)\simeq I_{H(G)}(P_1,\cV_1,Q_1)$, then $P=P_1,Q=Q_1$ and $\cV\simeq \cV_1$. 
\end{montheo}

    The proof   follows the same pattern as for the group $G$, by a descent   to $R$  of the classification over $R^{alg}$  \cite{AHenV2}. 
    
 In Proposition   \ref{prop:scalarLH}, we prove  that $I_{H(G)}(P,-,Q)$ and  scalar extension are compatible, as in the group case (Prop. \ref{prop:es2}).

    Assuming that $R$ contains a root of unity of order the exponent of $Z_k$ (the quotient of the parahoric subgroup of $Z$ by its pro-$p$ Sylow subgroup), the simple supersingular $H(G)_R$-modules are classified \cite{Oss}, \cite[Thm.1.6]{VigpIwss}; in particular  when $G$ is semisimple and simply connected, they have dimension $1$.  With Theorem \ref{thm:classHG}, we get a complete classification of the $H(G)_R$-modules.

Note that the ring $H(M)$ does not embed in the ring $H(G)$  and different inductions  from $\Mod_R(H(M))$ to $\Mod_R(H(G))$ are possible. We denote  $CI_P^{H(G)} :\Mod_R(H(M))\to \Mod_R(H(G))$   the parabolic coinduction functor and  $CI_{H(G)}(P,\cV,Q)$ the corresponding $H(G)_R$-module associated to an $R$-triple $(P,\cV,Q)$ of $H(G)$, used  in \cite{Abe}.  The classification theorem (Thm. \ref{thm:classHG}) can be  equivalently expressed   with $CI_{H(G)}(P,\cV,Q)$ instead of $I_{H(G)}(P,\cV,Q)$, as in the case where $R$ is algebraically closed \cite[Cor. 4.24]{AHenV2}. In the appendix we recall  results of Abe on the different inductions $\Mod_R(H(M)) \to \Mod_R(H(G))$ and their relations.

\subsection{}\label{I.9}  In \S\ref{S:5},   we give   applications (Theorems \ref{thm:vsd}, \ref{thm:lattice}, \ref{thm:adjoint}, \ref{thm:sssc}) of   the classification for $G$  (Thm \ref{thm:classG}) and for $H(G)$  (Thm \ref{thm:classHG}); they  were  already known when $R$ is algebraically closed, except for parts  (ii),(iii) of Theorem \ref{thm:lattice} below.  
 
\begin{montheo} {\rm [Vanishing of the smooth dual]} \label{thm:vsd}The smooth dual of an infinite dimensional irreducible admissible $R$-representation of $G$ is $0$.
\end{montheo} 

This  was proved by different methods when the characteristic of  $F$ is $0$ in \cite{Kohl} and when $R$ is  algebraically closed in  \cite[Thm.6.4]{AHenV2}.
In \S\ref {S:3.5} we deduce easily the theorem  from the  theorem over $R^{alg}$ using the scalar extension to $R^{alg}$ (Theorem \ref{thm:DA}).
  
\bigskip  [Description of   $\Ind_P^G \sigma$  for an irreducible admissible $R$-representation of $M$, and  of  $\Ind_P^{H(G)} \cV$ for a simple $H(M)_R$-module $\cV$]

We write  
  $\mathcal L_\pi$ for the lattice of subrepresentations of an $R$-representation $\pi$ of $G$,  and  $ \mathcal L_\cx$ for the lattice of submodules of an $H(G)_R$-module $\cx$.

  Recall that for a set $X$, an  upper set  in $\mathcal P(X)$ is a set $\mathcal Q$  of subsets of $X$, such that if $X_1\subset X_2 \subset X$ and $X_1\in \mathcal Q$ then $X_2\in \mathcal Q$.  Write $\mathcal L_{\mathcal P(X), \geq}$ for the lattice of upper sets in $\mathcal P(X)$.  For two subsets $X_1,X_2$ of $X$ write $X_1\setminus X_2$ for the complementary set of $X_1\cap X_2$ in $X_1$.
   
 By the classification theorems,   
$\sigma\simeq I_M(P_1\cap M, \sigma_1, Q\cap M)$  with  $(P_1, \sigma_1,Q)$ an $R$-triple of $G$, $Q\subset P$  and 
  $\sigma_1$ irreducible admissible supersingular  and    $\cV \simeq  I_{H(M)}(P_1\cap M, \cV_1, Q\cap M)$ with $(P_1, \cV_1,Q)$ an $R$-triple of $H(G)$, $Q\subset P$, and 
$ \cV_1$ simple supersingular.

\begin{montheo} \label{thm:lattice}   {\rm [Lattices  $\mathcal L_{\Ind_P^G\sigma}$ and $\mathcal L_{\Ind_P^{H(G)}\cV}$ ]}    
 
 (i)    $\Ind_P^G \sigma$    is multiplicity free of irreducible subquotients (isomorphic to) $I_G(P_1, \sigma_1, Q')$ for  $R$-triples $(P_1, \sigma_1, Q')$ of $G$ with $ Q' \cap   P=Q$ (notations as above).
    
      Sending $I_G(P_1, \sigma_1, Q')$   to $\Delta_{Q'} \cap (\Delta_{P(\sigma_1)}\setminus \Delta_{P})$ gives a lattice isomorphism\footnote{see the discussion in \cite{He} and \cite{AHHV} on the lattice of submodules of a multiplicity free module}
 $$\mathcal L_{ \Ind_P^G \sigma}\to \mathcal L_{ \mathcal P(\Delta_{P(\sigma_1)}\setminus \Delta_{P}), \geq}.$$

 (ii) The $H(G)_R$-module  $\Ind_P^{H(G)}\cV$ satisfies the analogue of (i).
    
(iii)  If
 $ \sigma^{I\cap M}$ is  simple and the natural map  $\sigma^{I\cap M}\otimes_{H(M)}\mathbb Z[(I\cap M)\backslash M]\to \sigma$ is injective, then 
 the  $I$-invariant functor $(-)^I$ gives a  lattice isomorphism 
$\mathcal L_{\Ind_P^G( \sigma)} \to\mathcal L_{\Ind_P^{H(G)} ( \sigma^{I\cap M})}$ with inverse given by  $-\otimes_{H(G)_R}R[I\backslash G]$.
\end{montheo}

When  $R$  is algebraically closed (i)  is proved  in \cite[\S 3.2]{AHenV1}.  In    \S \ref{S:54} we  prove (i) and   (ii);   (iii) follows from (i), (ii), Theorem \ref{thm:latticeA} and  the commutativity of the parabolic inductions with $(-)^I$ and  $-\otimes_{H(G)}\mathbb Z[I\backslash G]$ \cite{OV}.

 \begin{moncor} \label{cor:lattice}

1. \  The socle and the cosocle of $\Ind_P^G\sigma$ are irreducible; $\Ind_P^G\sigma$    is irreducible if and only if $P$ contains $P(\sigma_1)$.  The same is true  for $\Ind_P^{H(G)} \cV$.
 
2. \  Let $\pi $ be an irreducible admissible $R$-representation of $G$; write $\pi\simeq I_G(P, \sigma, Q)$ with $\sigma$  irreducible admissible supersingular. 

 If   
   $\sigma^{I\cap M}$ is simple and the natural map $\sigma\to \sigma^{I \cap M}\otimes_{H(M)}\mathbb Z[(I\cap M)\backslash M]$ is bijective, then $\pi^I$ is simple and 
$\pi\simeq \pi^I \otimes_{H(G)}\mathbb Z[I\backslash G] .$  \end{moncor} 
The first assertion for $\sigma$
and  $R$ is algebraically closed  is \cite[Cor. 3.3 and 4.4]{AHenV1}. The second assertion follows from Theorem \ref{thm:lattice} (iii).

\bigskip [Computation of the left and right adjoints of the parabolic induction, of  $\pi^I$ for  an irreducible admissible $R$-representation $\pi$ of $G$ and of
 $\cx \otimes_{H(G) }\mathbb Z[I\backslash G]$ for a simple $H(G)_R$-module $\cx$]

 For a parabolic subgroup $P_1$ of $G$, write  $L_{P_1}^G$ and $R_{P_1}^G$ for  the left and right adjoints of $\Ind_{P_1}^G$, and   $L_{P_1}^{H(G)} $ and $R_{P_1}^{H(G)} $  for the  left and right  adjoints of $\Ind_{P_1}^{H(G)}$ \cite{Vigadjoint}.  

 There is nothing new in Theorem \ref{thm:adjoint} below, now that we know that $\pi \simeq  I_G(P, \sigma, Q)$   with $\sigma$ irreducible admissible supersingular, and  $\cx \simeq I_{H(G)}(P, \cV, Q)$ with $\cV$ simple supersingular.  It suffices to quote:
   for   $R_{P_1}^G(\pi)$  \cite[Corollary 6.5]{AHenV1}, for $L_{P_1}^G(\pi)$   \cite[Cor. 6.2, 6.8]{AHenV1}, for       $L_{P_1}^{H(G)}(\cx)$ and $R_{P_1}^{H(G)}(\cx)$ (\cite[Thm. 5.20]{Abeparind} when $R$ is algebraically closed, but this hypothesis  is not used),  for 
$\pi^I$ and  $\cx \otimes_{H(G) }\mathbb Z[I\backslash G] $   \cite[Thm.4.17, Thm.5.11]{AHenV2}.

\begin{montheo} \label{thm:adjoint} {\rm [Adjoint functors  and   $I$-invariant]}  

\smallskip  (i) $L_{P_1}^G(\pi)$ and   $R_{P_1}^G(\pi)$ are $0$ or irreducible admissible.

 $L_{P_1}^G(\pi) \neq 0 \ \Leftrightarrow  \  {P_1}\supset P, \langle {P_1},Q\rangle \supset P(\sigma) \ \Leftrightarrow \ L_{P_1}^G(\pi)\simeq  I_{M_1}(P\cap M_1, \sigma, Q \cap M_1) .$ 

  $R_{P_1}^G(\pi) \neq 0\  \Leftrightarrow  \  {P_1}\supset Q\  \Leftrightarrow \ R_{P_1}^G(\pi)\simeq I_{M_1}(P\cap M_1, \sigma, Q  \cap M_1).$ 

\smallskip (ii)   $L_{H(M)}^{H(G)}(\cx)$ and $R_{H(M)}^{H(G)}(\cx)$ satisfy (i)  with the obvious modifications.

\smallskip (iii) $\pi ^I\simeq  I_{H(G)}(P, \sigma^{I\cap M}, Q)$ and 
$\cx \otimes_{H(G)_R}R[I\backslash G] \simeq I_G(P,\cV\otimes_{H(M)_R}R[(I\cap M)\backslash M] , Q)$.
\end{montheo} 
Example: $L^G_{P(\sigma)}( I_G(P,\sigma, Q))\simeq R^G_{P(\sigma)}( I_G(P,\sigma, Q))\simeq\St_Q^{M(\sigma)}(\sigma) $ and the analogous  for $ I_{H(G)}(P, \cV, Q)$.  

\bigskip [Equivalence between 
 supersingularity  and  more classical  notions  of cuspidality]
 
  An irreducible admissible  $R$-representation $\pi$ of $G$  is said to be  

 - {\bf supercuspidal} if it is not a subquotient of $\Ind_P^G\tau$ with $\tau\in \Mod_R^\infty(M)$ irreducible admissible  for all  parabolic subgroups $P=MN \subsetneq G$.
 
 - {\bf cuspidal} if $L_P^G(\pi)=R_P^G(\pi)=0$ for all  parabolic subgroups $P \subsetneq G$.

  \begin{montheo} \label{thm:sssc} $\pi$ is  supersingular if and only if $\pi$  is supercuspidal if and only if $\pi$ cuspidal.
  \end{montheo}
 
The equivalence of supersingular with supercuspidal, resp. cuspidal,  follows immediately from Theorem \ref{thm:lattice},  resp. Theorem \ref{thm:adjoint}.
When  $R$  is algebraically closed, the first equivalence  was proved in   \cite[Thm VI.2]{AHHV}   and the second one in \cite[Cor.6.9]{AHenV1}. 

\medskip Given an irreducible admissible $R$-representation $\pi$, there is a  parabolic subgroup $P=MN$ (containing $B$) and an irreducible admissible supercuspidal $R$-representation of $\sigma$ of $M$ such that $\pi$ is a subquotient 
of $\Ind_P^G(\sigma)$, $P$ and the isomorphism class of $\sigma$  (called the supercuspidal support  of $\pi$) are unique.

 \bigskip
 \noindent \textbf{Acknowledgments} 
We thank the  CNRS, the   IMJ Paris-Diderot University, the Paris-Sud University,  and the Weizmann Institute where part of our work was carried out.

\section{Some general algebra}\label{S:2}

\subsection{Review on scalar extension}\label{S:21}  We consider a field  $R$ and  an $R$-algebra $A$ (always associative with unit).

For  an extension  $R'$ of $R$ (which we see as a field $R'$ containing $R$), the scalar extension $R'\otimes_R - :\Mod_R\to \Mod_{R'}$ from $R$ to $R'$,    also denoted  $(-)_{R'}$, is faithful  exact and  left adjoint to the restriction  from $R'$ to $R$. 

The scalar extension $A_{R'}$ of $A$  is an $R'$-algebra and if $W$ a (left or right) $A$-module, 
 $W_{R'}$    is an $A_{R'}$-module. An  $A_{R'}$-module of the form $W_{R'}$ is said to descend  to $R$.
   
 \begin{remark}\label{rem:1} 
{\rm Let  $R^{alg}$ be an algebraic closure of $R$. If $A$ is a finitely generated $R$-algebra,   an $A_{R^{alg}}$-module $W$ of finite dimensional over $R^{alg}$ descends to a finite extension of $R$.   Indeed, 
if $(w_i)$ is an $R^{alg}$-basis of $W$, $(a_j)$ a finite set of generators of $A$, and $a_jw_i= \sum_k r_{j,i,k}w_k$, the extension $R'/R$ generated by the  coefficients $r_{j,i,k}\in R^{alg}$ is finite and $W$ is the scalar extension of 
the $A_{R'}$-module $\oplus_i R' w_i$.
 } 
 \end{remark}
 
 \begin{remark} \label{rem:2} If $V,W$ are $A$-modules, the natural map 
\begin{equation}\label{eq:sce}(\Hom_{A} (V,W))_{R'} \to \Hom_{A_{R'}} (V_{R'},W_{R'})   \end{equation}
is injective, and bijective if $R'/R$ is finite or if $V$ is a finitely generated $A$-module 
\cite[\S 12, n$^o$2 Lemme 1]{BkiA8}, \cite[II prop.16 p.110]{BkiA2}.
 \end{remark}

We assume from now on in \S \ref{S:2} that $V$ is a simple $A$-module (in particular finitely generated); we write $D$ for the commutant $\End_A(V)$, so that $D$ is a division algebra, and $E$ for the center of $D$. By Remark \ref{rem:2}, the commutant of $V_{R'}$ is $D_{R'}$ and its  center is $E_{R'}$ (consider $V$ as an $A\otimes_RD$-module. That  $V$  is simple has further consequences:

\smallskip  (P1) As an $A$-module $V_{R'}$  is  a  direct sum of  $A$-modules isomorphic to $V$, i.e. $V$-isotypic of type $V$ 
 \cite[\S 4, n${}^o$4, Prop.1]{BkiA8}.

\smallskip (P2) The map $J\mapsto  JV_{R'}$ is a lattice  isomorphism of the lattice  of  right ideals $J$ of $D_{R'}$ onto the lattice of $A_{R'}$-submodules  of $V_{R'}$, with inverse $W\mapsto \{d\in D_{R'}, dV_{R'}\subset W\}$  \cite[\S 12, n${}^o$2, Thm.2b)]{BkiA8}.

\smallskip  (P3)   The map $I\mapsto ID_{R'}$  is a lattice  isomorphism of the lattice  of  ideals $I$ of $E_{R'}$ onto the lattice of  two-sided ideals $J$ of $D_{R'}$, the inverse map sending  $J\mapsto J\cap E_{R'}$   \cite[\S 12, n${}^o$4, Prop.3a)]{BkiA8}. 

\smallskip  (P4) If $R'/R$ is finite, $V_{R'}$ has finite length as an $A$-module, so also as an $A_{R'}$-module; then $D_{R'}$ is left and right artinian and $E_{R'}$ is artinian \cite[\S 12, n${}^o$5, Prop.5]{BkiA8}. If moreover $R'/R$ is separable, $V_{R'}$ is semisimple \cite[\S 12, n${}^o$5, Cor.]{BkiA8}.

\smallskip  (P5) If $D$ has finite dimension over $R$, $D_{R'}$ has the same dimension over $R'$, and by (P2) $V_{R'}$ has finite length $\leq [D:R]$ over $A'$.

\bigskip In the reverse direction:

  \begin{monlem}\label{lem:sq'} Let  $R'/R$ be an extension and  $V'$ a simple $A_{R'}$-module  descending to a finite  extension of $ R $. Then  $\Hom_{A_{R'}}(V', V_{R'})\neq 0$ for some  simple $A$-module $V$. For any such $V$, $\dim_RV$ is finite if $\dim_{R'}V'$ is, and  $\dim_R \End_AV$ is finite if  $\dim_{R'} \End_{A_{R'}}V'$ is.
 \end{monlem} 
  
 \begin{proof}  
 a) Assume first that $R'/R$ finite. Then $A_{R'}$ is a (free) finitely generated $A$-module, so $V'$ as an $A$-module is finitely generated, and in particular has a simple quotient $V$: $\Hom_A(V',V)\neq 0$. By Remark \ref{rem:2}, $\Hom_{A_{R'}}(V'_{R'},V_{R'})\neq 0$; but $V'_{R'}$ is the sum of $[R':R]$ copies of $V'$ so $\Hom_{A_{R'}}(V',V_{R'})\neq 0$. 
 
 Let $V$ be any simple $A$-module with $\Hom_{A_{R'}}(V',V_{R'})\neq 0$. Then by the same reasoning $\Hom_{A }(V',V )\neq 0$ so $\dim_R V$ is finite if   $\dim_{R'}V'$ is. Put $D=\End_A(V)$ and $D'=\End_{A_{R'}}(V')$ and let $W$ be the maximal $V'$-isotypic submodule of  $V_{R'}$. Then $W$ is $D_{R'}$-stable and  we get an homomorphism $D_{R'}\to \End_{A_{R'}}W$ which is necessarily injective  on $D$, since $D$ is a division algebra. By  (P4), 
  $V_{R'}$  has finite length, so is $W$ and   $\End_{A_{R'}}W$ is a matrix algebra over  $D'$; it follows that if $\dim_{R'}D'$ is finite, so is $\dim_{R'}(\End_{A_{R'}}W)$ hence also  $\dim_{R}(\End_{A_{R'}}W)$ and  $\dim_{R}(D)$.

b) Let us treat the general case. By assumption there is  a finite subextension $L$ of $R$ in $R'$ and an $A_L$-module $U$ such that  $V' = R'\otimes_LU$ - then $U$ is necessarily simple.  By a) $\Hom_{A_L}(U,V_L)\neq 0$ for some  simple $A$-module $A$ and by Remark \ref{rem:2}, $\Hom_{A_{R'}}(V',V_{R'})\neq 0$.  

Conversely if $V$ is some   simple $A$-module with $\Hom_{A_{R'}}(V',V_{R'})\neq 0$ then by Remark \ref{rem:2} again $\Hom_{A_{L}}(U,V_{L})\neq 0$, so the other assertions follow from a).
 \end{proof}

\begin{remark}{\rm   A non-zero $A$-module $W$ is called  absolutely simple  if  $W_{R'}$  is  simple for   any extension $R'/R$. A simple $A$-module $V$ of commutant $D$   is  absolutely simple if and only if $D=R$. 
 For  $\Rightarrow$  \cite[\S3,n$^o$2,Cor.2, p.44]{BkiA8}. For $\Leftarrow$, the commutant of $V_{R'}$ is $R'$  and  P1) implies that $V_{R'}$ is simple.


If $R$ is algebraically closed of cardinal $> \dim_RV$,  then  $D=R$ \cite[\S3,n$^o$2,Thm.1, p.43]{BkiA8}.

There exists an algebraically closed  extension $R'/R$ of cardinal $> \dim_RV$ 
  \cite[\S3,n$^o$2, proof of Cor.3, p.44]{BkiA8}.
  }
  \end{remark}

 \subsection{A bit of ring theory}\label{ss:at}

\begin{monlem} \label{lem local} Let $L/K$ be a field extension and $E/K$ be a finite purely inseparable extension.   Then $L\otimes_KE$ is an artinian local ring with residue field $L$.
\end{monlem} 
 \begin{proof} This is probably well known but I do not know a reference. Here is a proof.
  
When $E=K$, this is obvious. Assume $E\neq K$. Then, the characteristic of $K$ is positive, say $p$. There is a finite filtration from $K$ to $E$ by subfields $K=E_0\subset \ldots   \subset E_{i} \subset \ldots \subset E_n=E$ such that $E_i\simeq E_{i-1}[X]/ (X^p)$. By induction on $i$, we suppose that $A_{i-1}:=L\otimes_KE_{i-1}$ is an artinian local ring with residue field $L$. We show that $A_i$ has the same property. Clearly $A_i$ is an  artinian commutative ring, hence a finite product of local (artinian) rings \cite[Cor. 2.16]{Eis}. We have $A_i\simeq A_{i-1}[X]/(X^p)$. The only idempotents in $A_i$ are trivial (if $P(X)\in A_{i-1}[X]$ is unitary and satisfies $P(X)^2\equiv P(X)$ modulo $(X^p)$ then $P(X)=1$) hence $A_i$ is local. As $A_{i-1}$ is a quotient of $A_i$ and $L$ is a quotient of $A_{i-1}$, the field $L$ is a quotient of $A_i$.
\end{proof}
 \subsection{Proof of the decomposition theorem (Thm.\ref{thm:DA} and Cor. \ref{cor:DA})}
 
 \
Let  $V$ a simple $A$-module of commutant $D$ of finite dimension over $R$,  $\delta^2$ the dimension  of  $D$ over its center $E$. 

We recall   that  
a finite extension $E'/E$ splits $D$, i.e. $E'\otimes_ED\simeq M(\delta, E')$,  if and only if $E'$  is isomorphic to   a maximal subfield of a matrix algebra over $D$ \cite[\S 15, n$^o$3, Prop.5]{BkiA8}.   We recall also that $D$
contains  a maximal subfield, which a separable extension  $E'/E$ of degree $\delta$ \cite[7.24 Prop]{CR} or \cite[lo.cit. and \S 14, n${}^o$7]{BkiA8}.

Let 
 $R'/R$ be a normal extension containing a finite Galois extension  $E'/E$ splitting $D$ (for example an algebraic closure $R^{alg}$ of $R$).
  
 For $i\in \Hom_R(E,R')$,  $R'\otimes_{i,E} E$ is a quotient field of $E_{R'}$ isomorphic to $R'$ so
 $$\prod_{i\in \Hom_R(E,R')}R'\otimes_{i,E} E$$
is a semi-simple quotient of $E_{R'}$ of dimension $[E_{sep}:R]$. It is equal to $E_{R'}$ if and only if  $E/R$ is separable:  $[E:R]=[E_{sep}:R]$. By the same argument,  $$R'\otimes_RE_{sep} =\prod_{j\in \Hom_R(E_{sep},R' )}R' \otimes_{ j, E_{sep}} E_{sep}.$$ 
Recall from \S\ref{ss:at} the bijective map $i\mapsto j=i|_{E_{sep}}:  \Hom_R(E,R')\to  \Hom_R(E_{sep},R')=J$. Tensoring   by $ - \otimes_{E_{sep}}E$, we get  a product decomposition  $$E_{R'}\simeq \prod_{j\in J}R' \otimes_{ j, E_{sep}}  E.$$
inducing   decompositions 
$$D_{R'}\simeq \prod_{j\in J}R' \otimes_{ j, E_{sep}} D, 
\quad V_{R'} \simeq \oplus _{j\in J}\,R' \otimes_{ j, E_{sep}} V,$$
where  $R' \otimes_{ j, E_{sep}} D$ is the commutant of the $A_{R'}$-module $R' \otimes_{ j, E_{sep}} V$. 
As  $R'$ contains a Galois extension $E'/E$  splitting $D$, 
 $$ R' \otimes_{ j, E_{sep}} D \simeq M(\delta, R' \otimes_{ j, E_{sep}}  E), \quad R'\otimes_{j,E_{sep} }V \simeq \oplus^{\delta} V_{R',j},$$
for an  $A_{R'}$-submodule $V_{R',j}$ of commutant isomorphic to $R'\otimes_{j,E_{sep}}E$. 
The first isomorphism implies the second one \cite[\S 6 n${}^o$7, cor.2, p.103]{BkiA8}. 
 To prove the first isomorphism,
we choose  $g\in \Hom_R(E',R')$ extending $j$ and $i$  and  we compute:

\noindent $R' \otimes_{ j, E_{sep}} D \simeq R' \otimes_{ j, E_{sep}}  E\otimes_ED \simeq  
R'^G \otimes_R R'_{sep} i(E) \otimes_ED
 \simeq  R'^G \otimes_R R'_{sep} i(E) \otimes_{g,E'} E'\otimes_ED$
 
\noindent  $\simeq  R'^G \otimes_R R'_{sep} i(E) \otimes_{g,E'} M(\delta,E')
 \simeq  M(\delta, R'^G \otimes_R R'_{sep} i(E) \otimes_{g,E'} E')\simeq M(\delta, R' \otimes_{ j, E_{sep}}  E)$.

For any extension $L/R'$, we still have $\Hom_R(E,L)=\Hom_R(E,R')\simeq J$, 
   $$E_{L}\simeq \prod_{j\in J}L \otimes_{ j, E_{sep}}  E, \ 
D_{L}\simeq \prod_{j\in J}L\otimes_{ j, E_{sep}} D, \
\ V_{L} \simeq \oplus _{j\in J}\oplus^\delta \, V_{L,j},$$
with  $  V_{L,j} = ( V_{R',j})_L $ and $\End_{A_L} V_{L,j}\simeq L\otimes_{j,E_{sep}}E$.  

 By \S\ref{ss:at}, $L \otimes_{ j, E_{sep}}  E$  is an Artinian local ring.   We deduce that the $A_L$-module  $V_{L,j}$ of commutant $L \otimes_{ j, E_{sep}}  E$ is indecomposable \cite[\S 2 n$^o$3 Prop.4]{BkiA8} of length $[E:E_{sep}]$ and that  its  simple subquotients   are all isomorphic to
 $$ L \otimes_{(L\otimes_{j,E_{sep}}E)}V_{L,j}.$$ 
The decomposition  of $D_{L}$  shows that there are no non-zero $ A_{L}$-homomorphism between $V_{L,j}$ and $V_{L,j'}$ if $j\neq j'$ (also between 
the  simple subquotients of $V_{L,j}$ and of $V_{L,j'}$).
  
The $A_L$-module $ V_{L,j}=L\otimes_{R'}  V_{R',j}$ and the $A_{R'}$-module
 $  V_{R',j} $ have the same length,   the  scalar extension to $L$ of the $A_{R'}$-module  $ R' \otimes_{R'\otimes_{j,E_{sep}}E}V_{R',j}$ is simple. This being true for all $L$, the simple subquotients of $V_{R'}$ are absolutely simple. The same is true for their scalar extension to $L$.
 Taking $R'=E'$,   the  $A_L$-subquotients of $V_{L}$ are defined over $E'$.

An $R$-automorphism $g$ of $R'$ induces an $R$-isomorphism $r'\otimes v\mapsto g(r') \otimes v:  R' \otimes_{ j, E_{sep}} V\to R' \otimes_{g\circ  j, E_{sep}} V$ for $i\in J$.  
This action which corresponds to the transitive action of $G$ on  $J$,  induces a transitive action of $G$ on the set of isomorphism classes $[V_{R',j}]$ of   $V_{R',j}$ for $j\in J$.

So   the map $[V]\to \Aut_R(R')[V']$ where $V'$ is a simple subquotient of $V_{R'}$ is well defined. It is 
  injective because $V'$ seen as an $A$-module is $V$-isotypic, and it is
 surjective  by Lemma \ref{lem:sq'}.  This ends the proof  of Thm.\ref{thm:DA}. $\square$

  \begin{remark}  {\rm In Thm.\ref{thm:DA} we note that  the   subquotients of  $V_{R'}$  descend  to the  finite Galois extension $E'/R$. }
 \end{remark}

 We prove now Corollary \ref{cor:DA}. 
 
   We choose algebraic closures $L^{alg}\supset R^{alg}$ of $L\supset R$ containing a finite Galois extension $E'/R$ splitting $D$. 
 
(i)  The  length of $V_L$ is less or equal to the  length of $V_{L^{alg}}$ and the length of $V_{L^{alg}}$ is
 $ \delta [E:R]$.

 (ii) The length of $V_{E'}$ is $ \delta [E:R]$.   The commutant of any subquotient  $W$
   of $V_{E'}$ is contained in the commutant of $W$ seen as an $A$-module,  $W$ is $V$-isotypic of finite length as an $A$-module, and the dimension of the commutant $D$ of $V$ is finite. Hence the dimension of the commutant of $W$ is finite.  The scalar extension from $E'$ to $L^{alg}$ induces a lattice  isomorphism  $\mathcal L_{V_{E'}}\to \mathcal L_{V_{L^{alg}}}$. Any subquotient of $V_{L^{alg}}$ is the scalar extension $W_{L^{alg}}$ from $E'$  to $L^{alg}$ of some subquotient $W$ of $V_{E'}$. The dimension of the commutant of $W_{L^{alg}}$ is equal to the dimension of the commutant of $W$. The    scalar extension $W'_{L^{alg}}$ from $L$ to $L^{alg}$ of a subquotient $W'$ of $V_L$ is a subquotient of   $V_{L^{alg}}$.  The dimension of the commutant of $W'$ is equal to the dimension of the commutant of $W'_{L^{alg}}$, hence is finite.
      
 \subsection{Proof of the lattice theorems (Theorems  \ref{thm:latticeA}, \ref{thm:latticeA2}, Corollary \ref{cor:latticeA2})}\

  Motivated by   the parabolic induction and  the pro-$p$ Iwahori invariant functor   when $R$ is a field of characteristic $p$, we prove the lattice theorem (Theorem  \ref{thm:latticeA}) which generalizes  \cite[Lemma 3.11]{AHenV1} to the setting of:
  
  -  an adjunction   
   $(F,G,\eta, \epsilon)$  where $F : \mathcal C\to  \mathcal D$ is  a functor   between abelian categories  of right adjoint $G$, $\eta:\id\to G\circ F$ is the unit and $\epsilon:F\circ G\to \id$ is the counit.
   
-   $W$ is a finite length object of  $\mathcal C$ \cite[Ex. 8.20, p. 205]{KS}. 
 
\bigskip Familiar notions for modules extend to abelian categories.
\begin{monlem}\label{lem:inter}  The     partially ordered set $\mathcal L_W$ of  subobjects of $W$  is a lattice, i.e. for any pair of subobjects $X,X'$ of $W$, the common    subobjects of $X$ and of $X'$ have  a largest element $X\cap X'$  and the subobjects of $W$ with $X$ and $X'$ as subobjects have a smallest element $X+X'$. Writing  $X\oplus X'$ for the direct sum we have an exact  sequence 
$$0\to (X\cap X') \to (X \oplus X')\to X+X'\to 0.$$
\end{monlem}

\begin{proof}  \cite[Prop. 2.2.4, II.5.4 Axiom 3, Notations 8.3.10, II.5.Ex. 8.20]{KS} \end{proof}

We prove now  Theorem \ref{thm:latticeA}.
 \begin{proof} Step 1.  As $G$ is left exact, it defines a map  of ordered sets   $  \mathcal L_{F(W)} \xrightarrow{G}  \mathcal L_{(G\circ F)(W)}$.  If  $Y_1\subset Y_2 \subset W$ and $X$ is the kernel of $ F(Y_1)\to F(Y_2)$, then $ G(X)$ is the kernel of $(G\circ F)(Y_1)\to (G\circ F)(Y_2)$. By a) $\eta_W$ is an isomorphism, so $\eta_Y$ is an isomorphism for all $Y\subset W$. So  $G(X)=0$. By b)  $X=0$ so 
 $F$ defines a map  $ \mathcal L_W\xrightarrow{F} \mathcal L_{F(W)} $.

The  composite  map $ \mathcal L_W\xrightarrow{F} \mathcal L_{F(W)} \xrightarrow{  G} \mathcal L_{(G\circ F)(W)}$ is an isomorphism by a).  So $ \mathcal L_W\xrightarrow{F} \mathcal L_{F(W)}$ is  injective.  
 
Step 2.  We prove Step 1 and c) implies b') and c').  By  Step 1  the image  of a Jordan-Holder sequence of $W$ by $F$ has length $\geq \ell(W)$. It has length $\ell(W)$ and is a Jordan-Holder sequence of $F(W)$ if and only if $F(Y)$ is simple for any irreducible subquotient $Y$ of $W$, i.e. c) holds.
Then  $ G(X)$ is simple  for any simple   subquotient $X$ of $F(W)$ by  a).  So Step 1 and c) imply b') and c').

We prove  in 
  Steps 3, 4, 5 that b') and c') imply   the injectivity of
 $\mathcal L_{F(W)} \xrightarrow{G} \mathcal L_{(G\circ F)(W)}$, therefore  $ \mathcal L_W\xrightarrow{F} \mathcal L_{F(W)}$ is a lattice isomorphism of inverse  $\mathcal L_{F(W)} \xrightarrow{\eta^{-1}\circ G} \mathcal L_W$. 

Step 3. Let $X$ be a non-zero subquotient of $ F(W)$ of length $\lg(X)$. We prove by induction on the  length   that b') implies $\lg( G(X))\leq \lg(X)$.   Let  $X\to X''$   a simple quotient of kernel $X'$.  Then $G(X')$ is the kernel of $G(X) \to G(X'')$ as $G$ is left exact. By c')  $G(X'')$ is simple and  by induction on the length, we get $\lg( G(X))\leq \lg( G(X'))+1\leq  \lg(X')+1=\lg(X)$.

Step 4. If  $X$ is a non-zero  subobject of $F(W)$, we prove that c') and Step 3 imply $\lg( G(X))= \lg(X)$. Seeing $X$ as the kernel of  the quotient map
$ F(W) \to F( W)/X $, $G(X)$ is the kernel of the quotient map  $(G\circ F) (W) \to G ( F( W)/X)$ by left exactness of $G$.  By Step 3,  $\lg(G(X))\leq \lg(X)$   and $\lg(G(F( W)/X))\leq \lg( F(W)/X)$. By c'), $\ell (W)= \ell (F(W)) $. So $\lg( G(X))= \lg(X)$ and  $\lg(G(F( W)/X))=\lg( F(W)/X)$.

Step 5.  Let $X,X'$ be  subobjects of  $F(W)$ such that $G(X)= G (X')$.  We show that Step 4 implies $X=X'$. A functor between abelian categories commutes with finite direct sums \cite[II.5. Axiom A3]{KS} and a right adjoint functor is  left exact \cite[II.6.20 p.137]{KS}. 
Applying $G$ to the exact sequence of Lemma \ref{lem:inter}, $G(X\cap X')$ is  the kernel of $G(X\oplus X')= G( X)\oplus  G (X')\to  G(X+X')$.  By Step 4 and  length count, the last map is surjective. But then $ G(X+X')= G(X)+  G( X')$. So $ G(X+X')= G(X)= G(X')$ and by  length preservation $X+X'= X=X'$.   

Step 6. We showed that the properties a), b) and c)  imply the properties a), b') and c'). Conversely, assume  the properties a), b') and c'). As $G$ is left exact and $F(W)$ has finite length by c'), b') implies that $G(X)\neq 0$ for any non-zero object of $F(W)$, hence b') and c') imply b). We showed that  a), b), b'), c')  imply
 that $ \mathcal L_W\xrightarrow{F} \mathcal L_{F(W)}$ is a lattice isomorphism, and in particular c).  Therefore, the properties a), b) and c) hold true.

Step 7. Assume  that $W$  satisfies a), b) and c), or equivalently a), b') and c') by Step 6. Clearly, a subobject of 
 $ W$ has finite length satisfies a), b) and c) and a quotient  $W\to W' $ has finite length and satisfies a), c); it satisfies also    b') as $F(W')$   is a  quotient of $F(W)$ and has finite length. As b') implies b), $W'$ satisfies a), b), c). \end{proof}

\noindent {\bf  Proof of Theorem \ref{thm:latticeA2}}  
 The first assertions of the theorem are \cite[\S4 n$^o$4 Prop. 3,  n${}^o$4 Thm.  2]{BkiA8}. If $Y$ is stable by $b_W$, it is clear that $Y\otimes_R V$ is stable by $b_W \otimes b_V$. Conversely, assume that $Y\otimes_R V$ is stable by $b_W \otimes b_V$; then for $y\in Y$ and $v\in V$,   $b_Wy \otimes b_Vv$ belongs to $W\otimes_RV$. Applying an $R$-linear form $\lambda$ on $V$ we see that $\lambda( b_Vv) b_Wy$ belongs to $Y$. If $b_V \neq 0$  we can choose $v\in V$ and $\lambda$ such that $\lambda( b_Vv)\neq 0$ and then  $b_Wy\in Y$.
 
 \bigskip 
 
\noindent {\bf  Proof of  Corollary \ref{cor:latticeA2}}  
 We apply   Theorem \ref{thm:latticeA2} with the endomorphisms $b_W$ and $b_V$ attached to elements $b$ of $B'$; if $b\in B$, $b_W=\Id_W$ so any $Y$ is stable by $b_W$ and if $b\in B'-B$ then $b_V\neq 0$ by hypothesis. The assertions of $A'$-linearity are straightforward to check on the action of $B'$.

\begin{remark}\label{rem:latticesq}{\rm Let  $F : \mathcal C\to  \mathcal D$ be  a functor   between abelian categories and  $W$   a finite length object of  $\mathcal C$  satisfying:
$$X \mapsto F(X):\mathcal L_W\to \mathcal L_{F(W)} \quad \text{is a lattice isomorphism}. $$
Then any subquotient  of $W$ satisfies the same property. Indeed, this is clear for a submodule $W'$ of $W$, and  $\ell(W')=\ell(F(W'))$. For any exact sequence $0\to W_1\to W_2 \to W_3\to 0$ in  $\mathcal C$ with $ W_2$ a subobject of $W$, the sequence  $0\to F(W_1)\to F(W_2) \to F(W_3)\to 0$ in  $\mathcal D$ is exact  by length count. Let $\mathcal L_{W_2}(W_1)$ denote the lattice of subobjects $Y$ of $W_2$ containing $W_1$.
The map $ Y \mapsto F(Y):\mathcal L_{W_2}(W_1)\to \mathcal L_{F(W_2)}(F(W_1)) $ is a lattice isomorphism. Taking the cokernels, it corresponds to a lattice isomorphism
 $Z \mapsto F(Z):\mathcal L_{W_3}\to \mathcal L_{F(W _3)} $.
 }\end{remark}
 
  \begin{remark}\label{rem:ss3}{\rm  \cite[Prop. 2.4]{Vigadjoint},  \cite[Thm. 8.5.8]{KS}:

  For any adjunction $(F,G,\eta,\epsilon)$ between two categories, 

-   $F$ is fully faithful if and only if the unit $\eta$ is an isomorphism,

-  $G$ is fully faithful if and only if the counit $\epsilon $ is an isomorphism,

-  the following equivalent  properties   imply that $F,G$ are quasi-inverse of each other:
   
      \quad  -  $F$ and $G$ are fully faithful,

\quad  -   $F$ is an equivalence,

\quad -  $G$  is an equivalence.

 }\end{remark}

 \section{Classification theorem for $G$}\label{S:3}

   \subsection{Admissibility,  $K$-invariant,    and scalar extension}\label{s:ira}

 \
 
 In this section \ref{S:3}, $R$ is any field and $G$ is a  locally profinite group.  An $R[G]$-module $\pi$  is smooth if $\pi= \cup_K \pi^K$ with $K$ running through the  open compact subgroups  of $G$, and  is admissible if it  is smooth and $\dim_R\pi^K$ is finite for all $K$. Note that $\End_{R[G]} \pi  \subset \End_R \pi^K$ if $\pi^K$ generates $\pi$ so that $\dim_R\End_{R[G]} \pi $ is finite if $\dim_R\pi^K$ is finite.

 The  category  $\Mod_R(G)$ of $R[G]$-modules and the subcategory $\Mod_R^\infty(G)$  of smooth $R[G]$-modules are abelian, but not  the additive subcategory $\Mod_R(G)^a$ of  admissible $R[G]$-modules.   
  The subcategory  $\Mod^K_R(G)$  of   $R[G]$-modules $\pi$ generated by $\pi^K$ is additive with  a  generator $R[K\backslash G]$ is not abelian in general  \footnote{If $\Mod^K_R(G)$ is abelian and $G$ second countable, $\Mod^K_R(G)$ is a Grothendieck category (same proof  than for $\Mod_R(G)$ \cite[lemma 3.2]{Vigadjoint})}.  The  commutant  of $R[K\backslash G]$ is the  Hecke $R$-algebra
 $$\End_{ R[G]}  R[K\backslash G] \simeq R\otimes_{\mathbb Z} H(G,K)=H(G,K)_R$$ (no index if  $R=\mathbb Z$). We have the abelian category $\Mod_R(H(G,K) )$ of 
   right $H(G,K)_R$-modules ($H(G,K)$-modules over $R$).   The functor
 $$\mathcal T:=- \otimes_{H(G,K)} \mathbb Z[K\backslash G]:\Mod_R(H(G,K)) \to \Mod_R (G)$$
of image $\Mod^K_R(G)$   is left adjoint to the   $K$-invariant functor
$(-)^K$. 

The unit  $\epsilon: \id_{\Mod_R(H(G,K))}\to (-)^K \circ \mathcal T$  and the counit   $\eta:\mathcal T \circ (-)^K\to \id_{\Mod^K_R(G)}$ of the adjunction correspond to  the natural maps $ \cx \xrightarrow{\epsilon_\cx} \mathcal T(\cx)^K, \epsilon_\cx(x)=x\otimes 1$   for   $x\in \cx\in \Mod_R(H(G,K))$ and $\mathcal T(\pi^K)\xrightarrow{\eta_\pi} \pi, \eta_\pi(v\otimes Kg)=gv $ for $g\in G, v\in \pi^K, \pi \in \Mod_R(G)$.

\begin{monlem} \label{lem:admK}(i) If $\pi$ is admissible and simple,  then  $\dim_R \End_{R[G]} \pi$ is finite. 

(i) Let $R'$ be an extension of $R$. 

 $\pi$ is admissible if and only if the scalar extension $\pi_{R'}$ of $\pi$ from $R$ to $R'$ is admissible.  

The  adjoint functors 
$\mathcal T$, $(-)^K$,  the unit $\eta$ and the counit $\epsilon$ commute with   scalar extension from $R$ to $R'$:
 $\mathcal T(\cx)_{R'}\simeq \mathcal T(\cx_{R'}), \ (\pi^K)_{R'}\simeq (\pi_{R'})^K,  \ (\eta_{\pi})_{R'}\simeq \eta_{\pi_{R'}}, \ (\epsilon_{\cx})_{R'}\simeq \epsilon_{\cx_{R'}}.$ 
 \end{monlem}
 
  \begin{proof} Clear.
  \end{proof}
We deduce that if the unit (resp. counit) of the adjunction  is an automorphism of  $\Mod_R(H(G,K))$ (resp. $\Mod^K_R(G)$), it is an automorphism for any subfield of $R$.  
Recalling  Remark \ref{rem:ss3}:

 \begin{monlem}\label{lem:eqK} If the  functor $( -)^K:\Mod _{R}^K(G)\to \Mod_{R}(H(G,K))$ over $R$
   is an equivalence, then it   is an equivalence for  any subfield of $R$.
\end{monlem}

\begin{remark}{\rm Assume that  $R$ is a field of characteristic $p$.  
   When $K$ is a pro-$p$-Iwahori subgroup the functor $(-)^K$ of Lemma \ref{lem:eqK} is an equivalence  if $G=GL(2,\mathbb Q_p)$ and $p\neq 2$, or if $G=SL(2,\mathbb Q_p)$.

Indeed,  for $GL(2,\mathbb Q_p)$ this is  proved  with the extra-hypothesis that  $R$ contains a $(p-1)$-th root of $1$ (\cite{O} plus  \cite{K}), that we can remove with Lemma \ref{lem:eqK}.
For $G=SL(2,\mathbb Q_p)$, see  \cite[Prop. 3.25]{OS}.}
  \end{remark}

 \subsection{Decomposition Theorem for $G$} 
 
 \
 
Let $G$ be a   locally profinite group and let $R$ be a field.   For   any  irreducible admissible $R$-representation $\pi$ of $G$, its  commutant   $D=\End_{R[G]}\pi$ has  finite dimension (Lemma \ref{lem:admK} (ii)), and for any extension $R'/R$, $\pi_{R'}$ is admissible (loc. cit.). Let $R^{alg}$ be an algebraic closure of $R$.

 \begin{montheo}\label{thm:DG}    a)  Let $\pi$ be an irreducible admissible $R$-representation $\pi$ of $G$, $D$ its commutant, $E$  the center of $D$ and $\delta$ the reduced degree of $D$ over $E$. Then $\pi_{R^{alg}} $ has length $\delta [E:R]$,  its simple subquotients are isomorphic to submodules; they are admissible, descend to a finite extension of $R$,  their commutant is $R^{alg}$ and their isomorphism classes form a single orbit under $\Aut_R(R^{alg})$.
 
 b) The map which to $\pi$ as above associates the $\Aut_R(R^{alg})$-orbit of the irreducible subquotients  of  $\pi_{R^{alg}} $  is a bijection from the set of (isomorphism classes of)  irreducible admissible $R$-representations of $G$ onto  the set of $\Aut_R(R^{alg})$-orbit of (isomorphism classes of)  irreducible admissible $R^{alg}$-representations of $G$ descending to some finite extension of $R$.
 
 \end{montheo}

  This is immediate  from Theorem   \ref{thm:DA}, if we remark that a submodule of the admissible representation $\pi_{R^{alg}}$ is also admissible. Of course we could apply the more precise results of Theorem   \ref{thm:DA} for intermediate extensions $R'/R$  and  Corollary \ref{cor:DA} (use Lemma \ref{lem:admK}):

\begin{moncor}  Let $L$ be an extension of $R$ and $\pi$ as in Theorem \ref{thm:DG}. Then  $\pi_L$ has length $\leq \delta [E:R]$, its simple subquotients are admissible.
  \end{moncor}
  \begin{proof}  Let $R^{alg} \subset L^{alg}$ be  algebraic closures of $R\subset L$.  The scalar extension of $\pi$ in $R^{alg}$ has length $\delta [E:R]$,  the irreducible subquotients of $\pi_{R^{alg}}$ are all absolutely irreducible (their commutant is $R^{alg}$). 
  Therefore  $\pi_L$ has length $\leq \delta [E:R]$.  Let $\tau$ be an   irreducible subquotient  of $\pi_L$. Then $\tau_{L^{alg}}$ is a subquotient of $\pi_{L^{alg}}$.
  All the subquotients of $\pi_{R^{alg}}$ are admissible, the same is true for those of  $\pi_{L^{alg}}$, hence  $\tau_{L^{alg}}$ is admissible. Applying Lemma \ref{lem:admK}, $\tau$ is admissible.
   \end{proof}
    
  \subsection{The representations  $I_G(P,\sigma,Q)$}\label{S:II.1} 

\ 

 From now on  $G$ is a $p$-adic reductive group and   $R$ is a field of characteristic $p$.  

As stated in the introduction, our base field $F$ is locally compact  non-archimedean of residue characteristic $p$. A linear algebraic group over $F$ is written with a boldface letter like $\bh$, and its group of $F$-points  by the corresponding ordinary letter $H=\bh(F) $. 
We fix an arbitrary connected reductive $F$-group $\bg$, a maximal $F$-split torus $\bf T$ in $\bg$ and a minimal $F$-parabolic subgroup $\bb$ of $\bg$  containing $\bf T$ ; we write $\bz$ for the centralizer of $\bf T$  in $\bg$ and $\bu$ for the unipotent radical of $\bb$, $\bg^{is}$ for the product of the isotropic simple components of the simply connected cover of the derived group of $\bg$; the image of $G^{is}$ in $G$ is the normal subgroup $G'$ of $G$ generated by $U$, and $G=ZG'$.  Let $\Phi^+$ denote the set of roots of $\bf T$ in $ \bu\cap \bm$ and  $\Delta\subset \Phi^+$ the set of simple  roots.

We  say that $P $ is a parabolic subgroup of $G$ and write $P=MN$ to mean that  $\bp$   is an $F$-parabolic subgroup of $\bg$  containing  $\bb$ , $\bm$ the Levi subgroup   containing $\bz$ and  $\bn$ the unipotent radical; so $\bp= \bm \bb= \bm \bn$, the parabolic subgroups $P$ of $G$ are in bijection $P\mapsto \Delta_P$ with the subsets   $\Delta$.  We have
 $G=M\langle ^GN\rangle$  for the normal subgroup   $\langle ^GN\rangle$ of $G$ generated by $N$.  
  For $J \subset \Delta$ we write  $P_J=M_JN_J$ for the corresponding  parabolic subgroup  such that $J= \Delta_{P_J}$;  for a singleton $J=\{\alpha\}$ we rather write $P_\alpha=M_\alpha N_\alpha$.   Set $P^{is}$ for the parabolic subgroup of $G^{is}$ of image $P\cap G'$ in $G$.  
  
  The smooth  parabolic induction $\Ind_{P }^G:\Mod_R^\infty(M)\to \Mod_R^\infty(G)$ is fully faithful, and admits  a right adjoint $R^G_P$ and a left adjoint $L_P^G$ \cite{Vigadjoint}. 
 
For a pair of  parabolic subgroups $Q\subset P$, write $\Ind_Q^M$ for $\Ind_{Q\cap   M}^M$  and consider   the Steinberg $R$-representation $\St_{Q}^{M}(R)$ of $M$,
 quotient of   $\Ind_{Q}^{M} (R) $   ($M\cap Q$ acts trivially on $R$)  by the sum $\sum_{Q'}\Ind_{Q'}^M(R) $, $Q' $ running through  the parabolic subgroups  of $G$ with  $ Q\subsetneq Q' \subset P$.  
   The $R$-representation $ \St_{Q}^M(R) $ of $M$  is irreducible and   admissible. 
   
   Writing $\St_Q^M=\St_Q^M(\mathbb Z) $, $\St_{Q}^{M}(R)\simeq R\otimes_{\mathbb Z} \St_Q^M$.   If 
   $P_2=M_2N_2$ is the parabolic subgroup corresponding to $ \Delta_P   \setminus \Delta_Q$, the inflation to   $M^{is}_{2}$ of the   restriction of $ \St_{Q}^M$   to $M'_2$ is $\St_{(Q\cap M_2)^{is}}^{M^{is}_2}(R)$   \cite[II.8 Proof of Proposition  and Remark]{AHHV}. This is true for all $R$ so $ \St_{Q}^M(R) $ as an $R$-representation of $M'_2$  is absolutely irreducible.   
 
   To an $R$-representation $\sigma$ of $M$   are associated the following  parabolic subgroups of $G$:
 
 a)  $P_\sigma=M_\sigma N_\sigma $ corresponding to    the set  $\Delta_\sigma$ of $\alpha\in \Delta\setminus {\Delta_M}$ such that  $Z\cap M_\alpha'$ acts trivially on $\sigma$.
 
 b) $P(\sigma)=M(\sigma)N(\sigma)$ corresponding to $\Delta(\sigma)=  \Delta_P\cup\Delta _\sigma$. There exists an extension  $e(\sigma) $ to $P(\sigma)$ of the inflation of $\sigma $ to $P$; it is trivial on $N(\sigma)$, write also $e(\sigma) $ for its restriction to $M(\sigma)$  \cite[II.7 Proposition and Remark 2]{AHHV}. For $P\subset Q \subset P(\sigma)$, the generalized Steinberg representation $ \St_{Q}^{M (\sigma)}(\sigma)$ of $M(\sigma)$  (\S\ref{I} \eqref{eq:IGst}) is  admissible, 
isomorphic to
$ e(\sigma)\otimes_\mathbb Z \St_{Q} .$ 

c)  $P_{min}=M_{min}N_{min}$ the minimal parabolic subgroup of $G$ contained in $P$   such that $\sigma$  extends  an $R$-representation $\sigma_{min}$  of $P_{min}$ trivial on $N_{min}$ \cite[Lemma 2.9]{AHenV1}, \cite[\S2.2]{AHenV2}. We have $\Delta(\sigma_{min})=\Delta(\sigma)$,   $e_Q(\sigma)=e_Q(\sigma_{min})$,  $\Delta_{\sigma_{min}}$ and $\Delta_{ \sigma_{min}}\setminus \Delta_{P_{min}}$ are orthogonal.
 So $M(\sigma)=M_{min} M'_{\sigma_{min}}$, $M_{min} $ normalizes $M'_{\sigma_{min}}$, and  $e(\sigma)$ is trivial on $M'_{\sigma_{min}}$.

  When $P(\sigma)=G$, write  $P_{min,2}=P_{\sigma_{min}}$ so $G=M_{min} M'_{min.2}$, $M_{min} $ normalizes $M'_{ min,2 }$. For a parabolic subgroup $Q$ of $G$ containing $Q$,  the action  of  $M'_2$  on $ e(\sigma)$ is trivial and is absolutely irreducible on $\St_Q^G(R)$.    
 
 \begin{madef}\label{def:es2}  An $R$-triple $ (P,\sigma,Q)$ of $G$ consists of   a parabolic subgroup $P=MN$of $G$,  a smooth $R$-representation $\sigma $ of $M$,     a parabolic subgroup  $Q$ of $G$ with  $P\subset Q\subset P(\sigma)$. 
 To an $R$-triple $ (P,\sigma,Q)$ of $G$ we attach the smooth  $R$-representation of $G$ $$I_G(P,\sigma, Q)=\Ind_{P (\sigma)}^G(  \St_{Q}^{M (\sigma)}(\sigma) ).$$
 \end{madef}

The representation $I_G(P,\sigma,Q)=I_G(P_{min},\sigma_{min}, Q)$  is admissible when  $\sigma$ is admissible \cite[Thm.4.21]{AHenV1}.

 \begin{maprop}\label{prop:tensorstG}Let  $ (P,\sigma,Q)$ be an $R$-triple of $G$  with $\sigma$ admissible of  finite length such that  for each  irreducible subquotient  $\tau$ of $\sigma$, $P(\sigma)=P(\tau)$ and $I_G(P,\tau,Q)$ is irreducible.  Then $P(\sigma)=P(\sigma')$ for any non-zero subrepresentation  $\sigma'$ of $\sigma$,  and   
 $ I_G(P,-,Q)$ induces a lattice isomorphism $\mathcal L_{\sigma}\to  \mathcal L_{I_G(P,\sigma,Q)}$.\end{maprop}

\begin{proof} Clearly  $P(\sigma)\subset P(\sigma')$. As $\sigma'$ has finite length, it contains an irreducible subrepresentation $\tau$.
 From  $P(\sigma)\subset P(\sigma')\subset P(\tau)$ and $P(\sigma)=P(\tau)$, we get $P(\sigma)=P(\sigma')$.

We are in the situation of 
 Theorem \ref{thm:latticeA2}  with $A=R[M'_\sigma]\subset A=R[M(\sigma)]$ with the basis given by $M(\sigma)$,  the $R$-representations $\St_Q^{M(\sigma)}(R)$ and $e(\sigma)$ of $M(\sigma)$.  
 So the natural maps 

$e(\sigma)\to \Hom_{RM'_\sigma}(\St_Q^{M(\sigma)}(R),St^{M(\sigma)}_Q(\sigma) )$ 

$\Hom_{RM'_\sigma}(\St_Q^{M(\sigma)}(R),St^{M(\sigma)}_Q(\sigma)) \otimes_R  St^{M(\sigma)}_Q(\sigma) \to   St^{M(\sigma)}_Q(\sigma)$
 
 \noindent are isomorphisms of $R$-representations of ${M(\sigma)}$, and we have   the  lattice isomorphism

 $\sigma' \mapsto  \St_Q^{M(\sigma)}(\sigma'): \mathcal L_{\sigma}\to \mathcal L_{ \St_Q^{M(\sigma)}(\sigma)}$.

\noindent  As \cite{Vigadjoint}  the functor $\Ind_{P(\sigma)}^G:\Mod_R(M(\sigma))\to \Mod_R(G) $   is exact  fully faithful of  right adjoint $R_{P(\sigma)}^G$,   the unit of the adjunction is an isomorphism (Remark \ref{rem:ss3}). The length   $\St_Q^{M(\sigma)}(\sigma)$  is equal to the  (finite) length of $\sigma$, and 
 its irreducible subquotients  are 
 $ \St_Q^{M(\sigma)}(\tau)$ for the irreducible subquotients $\tau$ of $\sigma$;
if 
 $\Ind_{P(\sigma)}^G ( \St_Q^{M(\sigma)}(\tau))=I_G(P, \tau,Q)$ is  irreducible for all $\tau$,    we are in the situation of   Theorem \ref{thm:latticeA} for $F=\Ind_{P(\sigma)}^G$ and $W= \St_Q^{M(\sigma)}(\sigma)$, so the map
 $\sigma'\mapsto I_G(P,\sigma',Q): \mathcal L_{\sigma}\to  \mathcal L_{I_G(P,\sigma,Q)}$   is a  lattice isomorphism.
\end{proof}
  
 \begin{remark}{\rm  $I_G(P, \sigma, Q)$  determines the isomorphism class of  $e(\sigma) $ because 
 $$e(\sigma) \simeq \Hom_{RM'_{\sigma}}(\St_{Q}^{P(\sigma)}(R), R_{P(\sigma)}^G(I_G(P, \sigma, Q))) $$ 
 (proof of Prop. \ref{prop:tensorstG} and   $R_{P(\sigma)}^G(I_G(P, \sigma, Q))\simeq   \St_{Q}^{P(\sigma)}(\sigma)$).

}\end{remark}

  We  check  now that  the different steps of the construction of  $I_G(P,\sigma,Q)$  commute with  scalar extension. 
  
Recalling that,  for any commutative rings $R\subset R'$, the  scalar extension from $R$ to $R'$ is the left adjoint of the scalar restriction from $R'$ to $R$ and that for a field extension $R'/R$  of characteristic $p$, the functor $\Ind_P ^G $ is fully faithful, we have:

   \begin{maprop}\label{prop:scalarL}  
   
   (i) For any  commutative rings $R\subset R'$, the parabolic induction commutes with the scalar restriction  from $R'$ to $R$ and with the scalar extension from $R$ to $R'$. Hence
  the  left (resp. right) adjoint of the parabolic induction commutes with scalar extension (resp. restriction).

(ii) Let $R'/R$ be an extension of fields of characteristic $p$. Let  $\sigma' \in \Mod_{R'}^\infty(M)$ and $ \pi\in \Mod_{R}^\infty(G)$ of scalar extension $\pi_{R'}$ from $R$ to $R'$ isomorphic to $\Ind_P ^G (\sigma')$.   Then   $\sigma'$ is  isomorphic to the scalar extension $(L_P^G\pi)_{R'} $ of $L_P^G(\pi)$  from $R$ to $R'$.   \end{maprop}

\begin{remark}{\rm On admissible representations  $R_P^G$ is the Emerton's ordinary functor. We do not know if the ordinary functor  admits a right adjoint  or if it commute with  scalar extension.}
\end{remark}

      \begin{maprop}\label{prop:es2} {\rm [Strong compatiblity of $I_G(P,-,Q)$ with scalar extension]}
      
  Let  $(P,\sigma,Q)$ be an $R$-triple of $G$.
  
\medskip   (i)  Let $R'/R$ be an extension. Then 
  
  $P(\sigma)= P(\sigma_{R'})$,  $(P,\sigma_{R'},Q)$ is an $R'$-triple of $G$. If   $\sigma$ is irreducible and $\sigma'$ is  a snon-zero ubquotient  of $ \sigma_{R'}$, then  $P(\sigma)=P(\sigma')$.
  
  $  (e(\sigma))_{R'}= e(\sigma_{R'}) , \  St_Q^{P(\sigma)}(\sigma)_{R'}\simeq St_Q^{P(\sigma)}(\sigma_{R'})$ and  $I_G(P,\sigma,Q)_{R'}\simeq I_G(P,\sigma_{R'},Q)$.

\medskip  (ii) Let $R'$ a subfield of $R$. If  $e(\sigma)$ or    $ \St_Q^{P(\sigma)}(\sigma)$ or $I_G(P,\sigma,Q)$  descends to $R'$, then $\sigma$ descends to $R'$. 

If  $e(\sigma)$, resp.  $ \St_Q^{P(\sigma)}(\sigma)$, resp. $I_G(P,\sigma,Q)$,   is the scalar extension of an $R'$-representation $\tau'$, resp. $\rho'$, resp. $\pi'$,  then $\sigma$ is the scalar extension  of the natural  $R'$-representation of $M$ on 
$\tau' $, resp. $\Hom_{RM'_{\sigma}}(\St_{Q}^{P(\sigma)}(R) , \rho') $,  resp. $ \Hom_{RM'_{ \sigma}}( \St_Q^{P(\sigma)}(R),L_{P(\sigma)}^G\pi') $.

      \end{maprop}

\begin{proof} (i) As an $R$-representation,   $\sigma_{R'}  $ is the direct sum of $R$-representations isomorphic to $\sigma$.  If $\sigma$ is irreducible, any subquotient $\sigma'$ of $\sigma_{R'}  $ is $\sigma$-isotypic. 
For   $\alpha\in \Delta-\Delta_P$,  $Z\cap M_\alpha'$ acts trivially on an $R'$-representation $\tau$ if and only if   it acts trivially on $\tau$ seen as an $R$-representation.  So $P(\sigma)= P(\sigma_{R'})$ (hence $(P,\sigma_{R'},Q)$ is an $R'$-triple of $G$), and  if $\sigma$ is irreducible  $P(\sigma)=P(\sigma')$.  It is clear from the definition that the  extension commutes with scalar extension $(e(\sigma))_{R'}= e(\sigma_{R'})$.

(ii)  If $I_G(P,\sigma,Q)=\pi'_R$, we have 
$ \St_Q^{P(\sigma)}(\sigma)\simeq  (L_{P(\sigma)}^G \pi')_R$ (Proposition \ref{prop:scalarL}).  If $  \St_Q^{P(\sigma)}(\sigma) \simeq \rho'_{R}$, we have 
$e(\sigma)\simeq    (\Hom_{R'M'_{\sigma}}( \St_Q^{P(\sigma)}(R') , \rho')_{R}$ because   $e(\sigma) \simeq \Hom_{RM'_{\sigma}}( \St_Q^{P(\sigma)}(R')_R,  \rho'_R )$ (proof of Prop. \ref{prop:tensorstG}).    If $e(\sigma)\simeq \tau_{R}$ then  $\sigma \simeq (\tau|_M)_{R}$ because  the restriction to $M$ commutes with  scalar extension. \end{proof}

\subsection{Supersingular representations}\label{S:ss}  \ {}

In the setting of \S\ref{S:II.1} with  an algebraically closed field $R$ of characteristic $p$, 
the definition of supersingularity for an irreducible admissible $R$-representation of $G$  \cite{AHHV} uses the Hecke algebras of the  irreducible smooth $R$-representations of the special parahoric subgroups of $G$. It is shown in  \cite{OV} that this definition is equivalent to a simpler one using only the pro-$p$ Iwahori Hecke $R$-algebra of   $G$. This latter definition has the advantage to   extend easily to the non-algebraically closed case.

Let $R$ be a field of characteristic $p$.  
 Let $I$ be a pro-$p$ Iwahori subgroup of $G$ compatible with $B$, so that $I\cap M$ is a pro-$p$ Iwahori subgroup of  $M$ for a parabolic subgroup $P=MN$ (recall that $M$ contains $Z$, and that the $M$ are parametrized by the subsets $J=\Delta_M$ of $\Delta$). The   pro-$p$ Iwahori Hecke ring $H(G,I) =H(G)$,  the  pro-$p$ Iwahori Hecke $R$-algebra  $H(G)_R$, the categories $\Mod_R(H(G))$ and $\Mod_R^\infty( G)$ are defined as in \S\ref{s:ira}.
   
 The group $Z_1=I\cap Z$ is the pro-$p$ Sylow subgroup of  the unique parahoric subgroup  $Z_0$  of $Z$ and $Z_k=Z_0/(I\cap Z)$ is  finite of order prime to $p$. The elements in $H(G)$ with support in $G'$ form a subring $H(G' ) $ normalized by a  subring of $H(G)$ isomorphic to  $\mathbb Z[\Omega]$ for a commutative finitely generated subgroup $\Omega$, and
 $$H(G )\simeq H(G' ) \, \mathbb Z[\Omega],\ H(G' ) \cap \mathbb Z[\Omega]\simeq \mathbb Z[Z'_k], \ Z'_k= (Z_0\cap G')/(I\cap G').$$  
 To $M$ is associated a certain element $T_M$ in $H(G' )$ which is central in $H(G)$ \cite{VigpIwss}.

 \begin{madef} \label{def:ss} 1. A (right) $H(G )_R$-module $\cx$ is called supersingular if, for all $M\neq G$,   the action of $ T_M$ on $\cx$  is locally nilpotent. 

2.  A smooth irreducible $R$-representation $\pi\in \Mod_R^\infty(G)$ is called supersingular if  the  right $H(G )_R$-module $\pi^I\in \Mod_R(H(G))$ is supersingular. 
\end{madef}

 When $\pi$ is admissible,  the definition is equivalent to the definition of 
\cite{AHHV} by \cite{OV}.

 \begin{remark}{\rm 
  1.  Let $0\to \mathcal V' \to\mathcal V\to \mathcal V''\to 0$ be an exact sequence of  $H(G)_R$-modules.  Then $\mathcal V$ is supersingular if and only if 
$ \mathcal V'$ and $\mathcal V''$ are supersingular.

2. When $R$ contains a root of unity of order the exponent of $Z_k$,
the simple supersingular $H(G )_R$-modules are classified \cite[Thm. 6.18]{VigpIwss}. As $H(G' )_R$-modules, they are sums of supersingular characters.

3.  The group $\Aut(R)$ of automorphisms of $R$ acts on $\Mod_R(G)$ and on  $\Mod_R(H(G))$. Clearly the action of $\Aut(R)$ commutes with the $I$-invariant functor, and respects supersingularity, irreducibility, and  admissibility.
  }
 \end{remark}

  We check easily that  supersingularity commutes   with scalar extension:       

  \begin{monlem}\label{lem:ss} Let $R'/R$ an extension,   $\cx\in \Mod_R(H(G))$, $\pi \in \Mod_R^\infty(G)$ irreducible   and  $\pi'$  an irreducible  subquotient of $\pi_{R'}$. Then $\cx$  is supersingular if and only if $\cx _{R'}$ is, 
 and
  $\pi$   is supersingular if and only if $\pi'$  is.
    \end{monlem}
  \noindent {\bf Proof}  As an $H(G)_R$-module, $\cx _{R'}$ is a direct sum of modules isomorphic to $\cx$. As an $R$-representation of $G$,
     $\pi'$ is a direct sum  of representations isomorphic to $\pi$.
 $\square$.

\begin{remark}\label{rem:tensorstG}{\rm   Let $\sigma$ be  an irreducible supersingular $R$-representation  of $M$.  The scalar extension $ \sigma _{R^{alg}}$   satisfies the conditions of Proposition \ref{prop:tensorstG}: 
 all irreducible subquotients $\tau$ of  $\sigma _{R^{alg}}$ are supersingular (Lemma \ref{lem:ss}),  $P(\tau)=P(\sigma)=  P(\sigma_{R^{alg}})$ (Prop.\ref{prop:es2} (i)), and $I_G(P,\tau,Q)$ is irreducible   (Classification theorem for $G$  over $R^{alg}$ \cite{AHHV}). }
\end{remark}

 \subsection{Classification of irreducible admissible $R$-representations of  $G$} \label{S:3.4} Let $R$ be a field of characteristic $p$ of algebraic closure $R^{alg}$. We prove in this section  the classification theorem for $G$  (Theorem \ref{thm:classG}). The arguments are  formal and rely on the properties:

{\bf 1}  The decomposition   theorem for $G$ (Thm.\ref{thm:DG}).

{\bf 2} The classification theorem for $G$ (Thm.\ref{thm:classG})  over $R^{alg}$  \cite{AHHV}.

{\bf 3}   The compatibility  of the scalar extension to  $R^{alg}$ with  supersingularity (Lem.\ref{lem:ss}) and the strong compatibility with $I_G(P,-,Q)$   (Prop.\ref{prop:es2}).  
 
 {\bf 4}  The lattice isomorphism  $\mathcal L_{\sigma _{R^{alg}}}\to \mathcal L_{I_G(P,\sigma _{R^{alg}},Q)}$ for the scalar extension $ \sigma_{R^{alg}}$ to $R^{alg}$ of an irreducible admissible supersingular  $R$-representation $\sigma$   (Prop.\ref{prop:tensorstG}  and  Rem.\ref{rem:tensorstG}).

\bigskip  We start the proof with an arbitrary  irreducible admissible $R$-representation $\pi$ of $G$.
By the decomposition theorem for $G$, the   scalar extension $\pi_{R^{alg}} $     has finite length;   we choose an   irreducible subrepresentation $\pi^{alg}$ of $\pi_{R^{alg}} $. By the decomposition theorem for $G$,  $\pi^{alg}$ is admissible, descends to a finite extension of $R$. By the classification theorem over $R^{alg}$, $$\pi^{alg}  \simeq I_G(P, \sigma^{alg},Q)$$ 
for an $R^{alg}$-triple
$(P=MN, \sigma^{alg},Q)$  of $G$ with  $\sigma^{alg} $   irreducible admissible supersingular. By  the strong  compatibility of  $I_G(P, -,Q)$ with scalar extension, 
$\sigma^{alg} $ descends to a finite extension of $R$.
By the decomposition theorem for $M$,
$ \sigma^{alg}$ is contained in   the scalar extension $\sigma_{R^{alg}}$  of an irreducible admissible $R$-representation $\sigma$. By  compatibility of the scalar extension  with supersingularity and  $I_G(P, -,Q)$,   $(P,\sigma, Q)$ is an $R$-triple of $G$,  $\sigma $  is supersingular and 
 $I_G(P,\sigma_{R^{alg}},Q)\simeq I_G(P,\sigma,Q)_{R^{alg}}.$ By the lattice isomorphism   $\mathcal L_{\sigma _{R^{alg}}}\to \mathcal L_{I_G(P,\sigma _{R^{alg}},Q)}$,
 $I_G(P,\sigma^{alg},Q) $ is contained in $I_G(P,\sigma _{R^{alg}},Q)$. As  an irreducible subquotient $\pi^{alg}$ of $\pi_{R^{alg}}$ is isomorphic to an irreducible subquotient $I_G(P,\sigma^{alg},Q)$ of $I_G(P,\sigma,Q)_{R^{alg}}\simeq  I_G(P,\sigma_{R^{alg}},Q) $, the decomposition theorem for $G$ implies that 
 $$\pi\simeq I_G(P,\sigma,Q) .$$ 
 
Conversely,
let 
$(P=MN, \sigma, Q)$ be an $R$-triple  of $G$ with  $\sigma $ irreducible admissible supersingular. We show that $I_G(P,\sigma, Q)$ is irreducible.   By the decomposition theorem for $M$,  $\sigma_{R^{alg}}$ has finite length,  $ I_G(P,\sigma_{R^{alg}}, Q)$ also by
  the lattice isomorphism  $\mathcal L_{\sigma _{R^{alg}}}\to \mathcal L_{I_G(P,\sigma _{R^{alg}},Q)}$,  $I_G(P,\sigma, Q)_{R^{alg}}  \simeq I_G(P,\sigma_{R^{alg}}, Q)$ and  the scalar extension is faithful and exact, hence $I_G(P,\sigma, Q)$ has also finite length. Let $\pi$ be an irreducible $R$-subrepresentation of $I_G(P,\sigma, Q)$. As  $I_G(P,\sigma, Q) $ is admissible, $\pi $ is admissible. The  scalar extension  $\pi_{R^{alg}}$  is isomorphic to a subrepresentation of $I_G(P,\sigma, Q)_{R^{alg}}  \simeq I_G(P,\sigma_{R^{alg}}, Q)$. By the lattice isomorphism  $\mathcal L_{\sigma _{R^{alg}}}\to \mathcal L_{I_G(P,\sigma _{R^{alg}},Q)}$,  $\pi_{R^{alg}}\simeq I_G(P,\rho, Q) $ for a subrepresentation $\rho$ of $\sigma_{R^{alg}}$. The representation $\rho$ descends to $R$ because $I_G(P,\rho, Q) $  does, by the strong compatibility of $I_G(P,-, Q) $ with scalar extension.  But  $\sigma_{R^{alg}}$ has no proper subrepresentation descending to $R$ by the decomposition theorem for $G$,  so $\rho=\sigma_{R^{alg}}$ and $\pi_{R^{alg}}= I_G(P,\sigma_{R^{alg}}, Q)\simeq I_G(P,\sigma, Q)_{R^{alg}} $, or equivalently,  $\pi\simeq I_G(P,\rho, Q)$.

Finally, let  $(P, \sigma, Q)$ and  $(P_1, \sigma_1, Q_1)$ be  two  $R$-triples  of $G$ with  $\sigma, \sigma_1 $ irreducible admissible supersingular and $ I_G(P,\sigma, Q) \simeq  I_G(P_1,\sigma_1, Q_1)$. By scalar extension 
  $ I_G(P,\sigma_{R^{alg}}, Q) \simeq  I_G(P_1,(\sigma_1)_{R^{alg}}, Q_1) $. The   classification theorem over $R^{alg}$ implies $P=P_1, Q=Q_1$ and some  irreducible subquotient $\sigma^{alg}$ of $\sigma_{R^{alg}}$ is isomorphic to some irreducible subquotient   $\sigma_1^{alg}$ of $(\sigma_1)_{R^{alg}}$. As $R$-representations of $G$, $\sigma^{alg}$ is $\sigma$-isotypic and $\sigma_1^{alg}$ is  $\sigma_1$-isotypic, hence
$\sigma, \sigma_1 $ are isomorphic.
 This ends the proof of the  classification theorem for $G$ (Theorem \ref{thm:classG}).

 \section{Classification theorem for $H(G)$}\label{S:4}

  As in \S\ref{S:II.1},  $G$ is a $p$-adic reductive group and   $R$ is a field of characteristic $p$. As in \S\ref{S:ss},  $I$ is a pro-$p$ Iwahori subgroup of $G$ compatible with $B$,  $H(G)$  is the   pro-$p$ Iwahori Hecke ring, $H(G)_R$ the  pro-$p$ Iwahori Hecke $R$-algebra, $Z_1$   the pro-$p$ Sylow of the unique parahoric subgroup $Z_0$ of $Z$, $Z_k=Z_0/Z_1$. In this section we prove for the right $H(G)_R$-modules the  results analogous to those of Section    \S\ref{S:3}. Although $H(G)$ and $G$ are related by 
  the $I$-invariant functor or its left adjoint, this relation  in characteristic $p$ does not satisfy the same properties than in the complex case and does not permit to  deduce the  case of the pro-$p$ Iwahori Hecke algebra  from  the  case of the group: the   similar results for $H(G)$ and $G$  have to be proved separately. 
  
   \subsection{Pro-$p$ Iwahori Hecke ring}\label{s:ring}

\ 
  
   The center $Z(H(G))$ of the pro-$p$ Iwahori Hecke ring $H(G)$  is a finitely generated ring and $H(G)$ is a finitely generated module over its center; the same is true  for  the $R$-algebra  $H(G)_R$ and  its center  $Z(H(G)_R)= Z(H(G))_R$ \cite{VigpIwc}.   This implies that the dimension over $R$ of a simple $H(G)_R$-module     is finite  \cite[2.8 Prop.]{Hn}.

Let $P=MN$  be a parabolic subgroup of $G$ (containing $B$ as in S\ref{S:II.1}).
The  pro-$p$ Iwahori Hecke ring $H(M)$  for the pro-$p$ Iwahori subgroup $I_M=I\cap M$ of $M$  (\S\ref{S:ss}) 
 does not embed in the ring $H(G)$. However we are in the good situation where  $H(M)$ is a localization of  a  subring  $H(M^+)$ (of elements supported in 
   the monoid $M^+:=\{m\in M \ | \ m(I\cap N)m^{-1}\subset  I\cap N\}$) which 
   embeds in $H(G)$.  We explain this in more details after introducing some notations \cite{VigpIw}.  
   
   An index $M$ indicates an object defined for $M$; for $G$ we supress the index.        We write  $\cn_M$  for the $F$-points of the normalizer of $\bf T$ in $\bm$,  $W_M= \cn_M/Z_1$,  $W_{M'}$ for the image of $M'\cap \cn_M$ in  $W_M$, $\Lambda=Z/Z_1$, $\ell_M$ for the length of $W_M$,  $\Omega_M$ for the image in $W_M$ of the  $m\in \cn_M$  normalizing $I_M $; $\Omega_M$ is the set of   $u\in W_M$ of length $\ell_M(u)=0$.
         
      The natural map  $W_M \to I_M\backslash M/I_M$ is bijective,  $W_M=W_{M'} \Omega_M, W_{M'}\cap  \Omega_M=  W_{M'}\cap Z_k$, $W_{M'}$ is a normal subgroup $W_M$ and a quotient of $W_{M^{is}}$ (via the quotient map $M^{is}\to M'$).    

For $m\in M$ and $w= w(m)\in W_M$ image of  $m_1\in \cn_M$ such that $I_Mm I_M= I_Mm_1 I_M$ (denoted also  $I_MwI_M$),  the characteristic function of $I_MmI_M$  seen as an element of $H(M)$  is written  $T^M(m)=T^M(w)$. We have also  the elements $T^{M,*}(m)=T^{M,*}(w)$  in  $H(M)$   defined by  $T^{M,*}(w)T^{M}(w^{-1})=[I_MwI_M:I_M] $  \cite[Prop.4.13]{VigpIw}.  For $u\in \Omega_M$, $T^*(u)=T(u)$ is invertible of inverse $T(u^{-1})$.  The $\Z$-module  $H(M)$ is  free of natural basis 
  $( T^M(w))_{w\in W_M}$  and of $*$-basis   $(T^{M,*}(w) )_{w\in W_M}$. 
 For the subring $H(M')$, the same is true with $W_{M'}$.
  The natural basis and the $*$-basis satisfy the same braid relations for $w_1,w_2\in W_M,\ell_M(w)= \ell_M(w_1)+\ell_M(w_2)$  and 
   the same quadratic relations  with a change of sign  for  $s\in W_{M'},  \ell_M(s)=1$:
  
  $ T^M(w_1) T^M(w_2)= T^M(w_1 w_2),\   T^{M,*}(w_1) T^{M,*}(w_2)= T^{M,*}(w_1 w_2)$,
  
 $T^M(s)^2= q_s + c_s T^M(s), \ T^M(s)^2= q_s - c_s T^M(s)$  
 
 \noindent where  $q_s=[I_MsI_M:I_M] \equiv 0  $ modulo $p$ and $c_s\in H(Z_0\cap M')$ (the subring of elements supported on $Z_0$),   $c_s\equiv -1$ modulo the ideal of $H(Z_0\cap M')$ generated by $p $ and $T(u)-1$ for  $u\in Z_k\cap W_{M'}$
  \cite{VigpIw}.  
 Both $q_s$ and $c_s$ do not depend on $M$ but $\ell_M$ depends on $M$,   The quotient map  $W_{M^{is}}\to W_{M'}$ respects the length and the coefficients of the quadratic relations,  the surjective natural linear map $H(M^{is})\to H(M')$  is a ring homomorphism   sending $T^{M^{is}}(w)$  to $T^M(w')$ and $T^{M^{is},*}(w)$  to $T^{M,*}(w')$ if $w\in W_{M^{is}}$ goes to $w'\in W_{M'}$ by the quotient map.

  The injective   linear maps  associated to the  bases  $$T^M(m)\mapsto T(m):H(M)\xrightarrow{\theta_M^G} H(G), \quad T^{M,*}(m)\mapsto T^*(m):H(M)\xrightarrow{\theta_M^{G,*}} H(G),$$ generally do not respect the product  
but their  restrictions  to  the subrings $H(M^+)$  and $H(M^-)$ (of elements supported on    the inverse monoid  $M^-$  of $M^+$)  do.  

\begin{remark}\label{rem:QQ1}{\rm 1. 
 For  $P\subset Q=M_QN_Q$, we have inclusions: if $\epsilon\in \{+,-\}$, then  
 
 $M^\epsilon\subset M_Q^\epsilon,\quad \theta_M^{G}(H(M^\epsilon))\subset \theta_{M_Q}^{G}(H(M_Q^\epsilon)), \quad \theta_M^{G,*}(H(M^\epsilon))\subset \theta_{M_Q}^{G,*}(H(M_Q^\epsilon))$.

2. When  $\Delta_M$ is orthogonal to $ \Delta \setminus \Delta_M$ the situation is simpler. Denoting  by $P_2=M_2N_2$ the  parabolic subgroup of $G$ corresponding to  $ \Delta \setminus \Delta_M$, we have:
    $G'$ is the direct product of $M'$ and of $M'_2$,  $M' \subset M^+$ and  $M'_2\subset M_2^+$, $G=M M'_2$,  $W=W_{M'_{2}}W_{M' }\Omega$ and $W_{M'_{2}}\cap W_{M' }\Omega= W_{M'_{2}}\cap Z_k$, the length  $\ell=\ell_G$  is equal to $\ell_M$ on $W_{M'}$ and to   $\ell_{M_2}$ on $W_{M'_2}$.
     For  $w \in W_{M'}, w_2 \in W_{M'_2}, u\in \Omega$,
$\ell(w )+\ell(w_2)=\ell(ww_2u)$. The braid and quadratic relations  satisfied by $T(w)$ for $w\in W_M$ are the same than for $T^M(w)$,  also for the $*$-basis,  and for $M_2$. 
    The maps $\theta_M^G$ and $\theta_M^{G,*}$  (resp. $\theta_{M_2}^G$ and $\theta_{M_2}^{G,*}$) are equal, respect the product, and $ H(M')\times H(M'_2)\xrightarrow{\theta_M^G \times \theta_{M_2}^G} H(G')$ is  a  ring isomorphism. 
    }\end{remark}

\subsection{ Parabolic induction $\Ind_{P}^{H(G)}$}\label{S:V3} 

\

 The  parabolic induction functor: 
\begin{equation}\label{eq:indH}\Ind_P^{H(G)}:=- \otimes_{H(M^+),\theta_M^G}H(G): \Mod_R(H(M))\to \Mod_R(H(G))
\end{equation}
corresponds via the pro-$p$ Iwahori invariant functor to  $\Ind_P^G:\Mod^\infty_R(M)\to \Mod^\infty_R(G)$ \cite[Prop.4.4]{OV}:
 $$(- )^I \circ \Ind_P^G \simeq \Ind_P^{H(G)} \circ (-) ^{I\cap M}: \Mod^\infty_R(M)\to \Mod^\infty_R(H(G)).$$
The parabolic induction functor   $\Ind_P^{H(G)}  $ has a right adjoint $R_{P}^{H(G)}$  and a   left adjoint $L_P^{H(G)}$ \cite{VigpIwst}. The right adjoint functors  of $\Ind_P^G$ and   of $\Ind_P^{H(G)}$ correspond via pro-$p$ Iwahori invariant functor 
 $$ (-)^{I\cap M} \circ R_P^G \simeq R_P^{H(G)}\circ (-^I ): \Mod_R(G)\to \Mod^\infty_R(H(M)), $$
but   the left adjoint functors do not \cite{OV}.  
As  $ 
 - \otimes_{H(M^+),\theta}H(G) \simeq   \Hom_{H(M^+),\theta^*}(H(G), -)$ (Proposition \ref{prop:exc} in the appendix below),
  the left and right adjoints of  $\Ind_P^{H(G)} $ are
\begin{align}\label{eq:RL} L_{P}^{H(G)}\simeq-\otimes_{H(M^+),\theta_M^{G,*}}H(M), \quad   
 R_{P}^{H(G)}= \Hom_{H(M^+),\theta_M^G }(H(M), -) .\end{align}
\begin{remark}{\rm  For the pro-$p$ Iwahori Hecke algebra, the left adjoint $L_{P}^{H(G)}$ being a localization is exact but for the group, the  left adjoint $L_P^G$ is not.  }
\end{remark} 

  \begin{maprop}\label{prop:RI} Let $R$ be a field of characteristic $p$. For two parabolic subgroups $P=MN, P_1=M_1N_1$ of $G$ (containing $B$),
  
(i)    $R_{P_1}^{H(G)} \circ \Ind_P^{H(G)}\simeq  \Ind_{P\cap P_1}^{H(M_1)} \circ  R_{P\cap P_1}^{H(M)}$.
 
(ii)  $L_{P_1}^{H(G)} \circ \Ind_P^{H(G)}\simeq  \Ind_{P\cap P_1}^{H(M_1)} \circ  L_{P\cap P_1}^{H(M)}$. 
  
(iii)  The parabolic induction  $\Ind_P^{H(G)}$ is fully faithful.

 \end{maprop}
 
\begin{proof} (i)  is equivalent to the same relation for the parabolic coinduction and its right adjoint, which is proved in  \cite[Prop. 5.1]{Abeparind}. 
What we call parabolic coinduction is  denoted by $I_P$ in \cite[\S 4]{Abeparind} (and called  parabolic induction).   The equivalence follows from the isomorphism   \cite[Thm.1.8]{VigpIwst},  \cite[Prop.2.21]{Abeparind} :
$$I_{P^{op}}\circ n_{w_Mw_G}\simeq  \Ind_{P }^{H(G)}    $$ 
where   

$w\mapsto n_w:\mathbb W \to W$ is an injective homomorphism from the Weyl group $\mathbb W$ of $\Delta$ to $W$ satisfying the braid relations (there is no canonical choice),  $w_M $ is  the longest element of the    Weyl group of $\Delta_M=\Delta_P$ for any parabolic subgroup $P=MN$ of $G$.

$P^{op}=M^{op}N^{op}$ denotes the parabolic subgroup of $G$ (containing $B$) with $\Delta_{M^{op}}=\Delta_{P^{op}}= w_Gw_P(\Delta_P)=w_G(-\Delta_P)$ (image of $\Delta_P$ by  the opposition involution  \cite[1.5.1]{T}).  

The  ring isomorphism  \cite[\S4.3]{Abe} 
$$H(M)\to H(M^{op}) \quad  T^M_w \mapsto  T^{M^{op}}_{n_{w_Gw_M} w n_{w_Gw_M}^{-1}} \quad\text{ for} \  w\in W_M, 
$$
induces by functoriality  a functor $\Mod_R (H(M^{op}))\xrightarrow{n_{w_Gw_M} (-)}\Mod_R (H(M))$ of inverse 

\noindent $\Mod_R (H(M))\xrightarrow{n_{w_Mw_G} (-)}\Mod_R (H(M^{op}))$, as $n_{w_Gw_{M^{op}}} =n_{w_P w_G}=n_{w_G w_P}^{-1}$.
  
(ii) This follows from (i) by  left adjunction and exchanging $P,P_1$.

(iii)  This follows from (ii) when $P_1=P$. As the functorial morphism $L_{P}^{H(G)} \circ \Ind_P^{H(G)}\to \id$ is an isomorphism,  $\Ind_P^{H(G)}$ is fully faithful \cite[Prop.2.4]{Vigadjoint}.
\end{proof}

 \subsection{The $H(G)_R$-module $ St_{Q}^{H(G)} (\cV) $}\label{S:V2} 

\

 Let  $\cV$ be 
a non-zero  right $H(M) $-module over $R$. We denote  $P_{\cV}=M_{\cV}N_{\cV},\ P(\cV)=M(\cV) N(\cV)$ the parabolic subgroups   of $G$ corresponding to
 \cite{Abe}   \cite[Def.4.12]{AHenV2}:
 $$ \Delta _{\cV}=\{\alpha \in \Delta \ \text{ orthogonal to  }\Delta_M ,  
v=v T^{M,*}(z)  \ \text{for all } \  v\in \cV,z\in Z\cap  M'_\alpha\},$$
$\Delta(\cV)= \Delta_M\cup\Delta_{\cV}$. The orthogonality of $\Delta_M$ and $\Delta_\nu$ implies that  $M(\cV)=M M_{\cV}$.  
  
 There is a right $H(M(\cV))_R$-module    
   $e(\cV)$  equal to $ \cV$ as an $R$-vector space, where $T^{M(\cV),*}(m)$ acts  like $T ^{M,*}(m)$ for $m \in M$ and acts trivially for $m\in M'_{\cV}$  \cite[Def.3.8 and remark before Cor. 3.9]{AHenV2}, called the   extension  of $ \cV$ to $H(M(\cV))$; we say  that  $\cV$ is the restriction of  $e(\cV)$ to $H(M)$. 
 For $P\subset Q=M_QN_Q \subset P(\cV)$, we define similarly the extension $e_Q(\cV)$ of   $ \cV$ to $H(M_Q)$. 
  
\begin{remark} {\rm The extension to  $H(M(\cV))$ gives a lattice isomorphism $\mathcal L_{\cV}\to \mathcal L_{e(\cV)}$.
 }
    \end{remark}

\begin{monlem} \label{lem:tensorstHG} Assume that $\Delta_M$ is orthogonal to $ \Delta \setminus \Delta_M$ and let $P_2=M_2N_2$ correspond to $\Delta_2:= \Delta \setminus \Delta_M$.

  For  any right $H(G)_R$-module    $\cx $ extending an $H(M)_R$-module and any right $H(G)_R$-module    $\cy $ extending an
$H(M_2)_R$-module, there is a structure of right $H(G)_R$-module  on

-   $\cx\otimes_R\cy$ where the $*$ basis  of $H(G)$  acts diagonally,

-  $\Hom_{\theta^*(H(M'_2))}(\cy ,\cx\otimes_R\cy)$, where  $T^*(w)$ acts by   $(T^*(w)_\cx \otimes T^*(w)_{\cy}) \circ - \circ (T^*(w)_{\cy})^{-1}$ for $w\in  W_{M' }\Omega$ and trivially for $w\in W_{M'_{2}}$.

\end{monlem}

\begin{proof}   When  $\Delta_M$ is orthogonal to $ \Delta \setminus \Delta_M$, on an  $H(G)_R$-module    $\cx $ extending an $H(M)_R$-module,  the action $T^*(w)_{\cx}$  of $T^*(w)$ for $w\in W_{M'}$  is trivial,  hence $T^*(wu)_{\cx}=T^*(u)_{\cx}$ is invertible for $ u\in  \Omega$ (\S\ref{s:ring}).

For   $\cx\otimes_R\cy$ see \cite[Prop.3.15]{AHenV2},  \cite[Cor.3.17]{AHenV2}.

For  $\cz=\Hom_{\theta^*(H(M'_2))}(\cy ,\cx\otimes_R\cy)$, we check that the $T^*(w)_\cz$ for $w\in W$ respect the braid and quadratic relations (\S\ref{s:ring}).  The braid relations follow from  $W=W_{M'_{2}}W_{M' }\Omega$
 and 
$T^*(ww_2u)=T^*(w) T^*(w_2) T^*(u) $  if   $w \in W_{M' }, w_2\in W_{M'_2}, u\in \Omega$. For the quadratic relations, let $s_2\in W_{M'_2}$ and $s\in W_{M'}$ of length $1$. Then $T^*(s_2)_\cx ,   T^*(s_2)_\cz $ and $T^*(s)_\cy $ are the identity.  As  $T^*(s_2)_\cz^2 =- T^*(s_2)_\cz$ and  $-c(s)_{\cz}$ is the identity, $T^*(s_2)_\cz$ satisfies the quadratic relation. 
As $T^*(s)_\cz(f)= (T^*(s)_\cx \otimes \id_\cy)(f)$ for $f\in cz$,  $T^*(s)_\cz$ satisfies the quadratic relation.
 \end{proof}

  The right $H(G)$-module $\St_{P }^{H(G)}: =(\St_{P }^G )^{I}$  is called a Steinberg $H(G) $-module and $ \St_{P }^G(R):=R\otimes_{\mathbb Z} \St_{P}^G$ a Steinberg $H(G) _R$-module. When $ P\subset Q=M_QN_Q$,     we write $\St_P^{H(Q)}:=\St_{P\cap M_Q}^{H(M_Q)} $; recall that  $(\Ind_Q^G (R))^I\simeq \Ind_Q^{H(G)}(R)$. 
   It is known that  \cite{Ly}:
  
-  $\St_{P }^{H(G)}(R)$ is absolutely simple and   isomorphic to the cokernel of the natural map
\begin{equation}\label{eq:st}\oplus_{P\subsetneq Q\subset G}  (\Ind_Q^G (R))^I\to  (\Ind_P^G (R))^I.
\end{equation}

-     
  $T^{*}(z)$ acts trivially on $\Ind_P^{H(G)}(\mathbb Z), \St_{P }^{H(G)}$ for  $z\in Z\cap M' $ \cite[Ex.3.14]{AHenV2}.

  - When $P(\cV)=G$,  $P_2=M_2N_2$ as in  Lemma \ref{lem:tensorstHG} and $P\subset Q$, we have  $ (\St_{Q }^G )^{I}=
  (\St_{Q }^G )^{I\cap M'_2}$  
   \cite[\S4.2, proof of theorem 4.7]{AHenV2}, 
   $$e(\cV) \otimes_R\Ind_Q^{H(G )}(R),\  e(\cV) \otimes_R\St_Q^{H(G )}(R), \ \Hom_{\theta^*(H(M'_2))}(e(\cV) ,\St_Q^{H(G )}(\cV))$$
are  right $H(G)_R$-modules  for the diagonal action of  $(T^*(w))_{w\in W}$ on  the first two ones,  and   for $T^*(w)$ acting  on  the other one by   $T^*(w)  \circ - \circ (T^*(w)_{e(\cV)})^{-1}$ for $w\in  W_{M' }\Omega$ and by the identity for $w\in W_{M'_{2}}$ (Lemma \ref{lem:tensorstHG}). 
We have an $H(G)_R$-isomorphism (\cite[Prop.4.5]{AHenV2} where it is explicited):
\begin{equation}\label{eq:indH} \Ind_{Q}^{H(G)}(e_{Q}(\cV))\simeq  e(\cV)\otimes (\Ind_{Q}^{G} (R))^I .
\end{equation}
These isomorphisms for  $P\subset Q\subset Q_1$ and the inclusion $(\Ind_{Q_1}^{G} (R))^I \subset (\Ind_{Q}^{G} (R))^I $ give an injective $H(G)_R$-isomorphism
$ \Ind_{Q_1}^{H(G)}(e_{Q_1}(\cV))\xrightarrow{\iota^G(Q,Q_1)}  \Ind_{Q}^{H(G)}(e_{Q}(\cV))$.

The cokernel  $\St_Q^{H(G )}(\cV)$ of the $H(G)_R$-map \begin{equation}\label{eq:IcV}\oplus_{Q\subsetneq Q_1\subset G} \Ind_{Q_1}^{H(G)}(e_{Q_1}(\cV))\xrightarrow{\oplus_{Q\subsetneq Q_1\subset G}\,\iota^G(Q,Q_1)} \Ind_{Q}^{H(G)}(e_Q(\cV))\end{equation}\label{eq:st}   is isomorphic to $e(\cV) \otimes_R\St_Q^{H(G )}(R)$ \cite[Cor.3.17, Cor.3.6]{AHenV2}.

  \begin{maprop}\label{prop:tensorstHG}  Assume  $P(\cV)=G$ and  $P\subset Q$. 
 
 1. The    natural maps   
    $e(\cV) \to  \Hom_{ H(M'_2)_R}(\St_Q^{H(G)}(R), \St_Q^{H(G)}(\cV))$
   and    
   
   $\Hom_{ H(M'_2)_R} (\St_Q^{H(G)}(R), \St_Q^{H(G)}(\cV)) \otimes_{\mathbb Z}\St_Q^{H(G )} \to\St_Q^{H(G)}(\cV)$  
   are $H(G)_R$-isomorphisms. 

2. The map
$Y\mapsto Y\otimes_R \St_Q^{H(G)}(R):\mathcal L_{e(\cV)}\to \mathcal L_{ \St_Q^{H(G)}(\cV)}$
is a lattice isomorphism of inverse $X \to \{y\in e(\cV), \ y \otimes_{\mathbb Z}\St_Q^{H(G )} \subset X\} $.
\end{maprop}
\begin{proof}We are in the setting of  Theorem \ref{thm:latticeA2} for $A=H(M'_2)_R\simeq \theta(H(M'_2)_R)\subset A'=H(G)_R$ with the  $* $- basis,  the right $H(G)_R$-module $\St_Q^{H(G)}(R)=e(\St_{Q}^{H(M_2)}(R))$, absolutely simple as an $\theta(H(M'_2)_R)$-module where $T^*_w$ for $w\in W \setminus W_{M'_2}$  (contained in $W_{M' }\Omega$) acts invertibly, and  the right $H(G)_R$-module $e(\cV)$ where $T^*_w$ for $w\in W_{M'_2}$ acts by the identity.
\end{proof}

 \subsection{The module $I_{H(G)}(P,\cV, Q)$}
\label{S:V4} 

\

   \begin{madef}\label{def:triH} An $R$-triple $ (P,\cV,Q)$ of $H(G)$ consists of   a parabolic subgroup $P=MN$ of $G$,  a right $H(M)_R$-module $\cV$,  $Q$  a parabolic subgroup of $G$ with  $P\subset Q\subset P(\cV)$. 
 To an $R$-triple $ (P,\cV,Q)$ of $H(G)$ corresponds  a  right $H(G)_R$-module 
 $$I_{H(G)}(P,\cV, Q)= \Ind_{P(\cV)}^{H(G)}( \St_{Q}^{H(M (\cV))}(\cV))$$  
 isomorphic to the cokernel of the $H(G)_R$-homomorphism
 $$\oplus_{Q \subsetneq Q_1\subset P(\cV)}\Ind_{Q_1}^{H(G)} (e_{Q_1}(\cV))\xrightarrow{\oplus_{Q \subsetneq Q_1\subset P(\cV)}\,\iota^G(Q_1,Q)}\Ind_{Q}^{H(G)} (e_{Q}(\cV))$$ 
where  $\iota^G(Q_1,Q)=\Ind_{P(\cV)}^{H(G)} (\iota^{M(\cV)}(Q\cap M(\cV), Q_1\cap M(\cV))).$
  
\end{madef}

     We can recover $e(\cV)$ from $I_{H(G)}(P,\cV, Q)$ and $P (\cV)$:  
\begin{equation} \label{eq:ecv}e(\cV)\simeq   \Hom_{H(M'_{\cV}),\theta^* }( \St_{Q}^{H(M(\cV))}(R), L_{H(M(\cV))}^{H(G)}(I_{H(G)}(P,\cV, Q)))
\end{equation}
by Proposition \ref{prop:tensorstHG} and   Proposition \ref{prop:RI}(iii):
\begin{equation} \label{eq:Lst}L_{H(M(\cV))}^{H(G)}(I_{H(G)}(P,\cV, Q)))\simeq \St_{Q}^{H(M (\cV))}(\cV) 
\end{equation}

 \begin{maprop}\label{prop:tensorstH} Let $ (P,\cV,Q)$ be an $R$-triple of $H(G)$ with $\cV$  of  finite length  and such that  for each  irreducible subquotient  $\tau$ of $\cV$, $P(\cV)=P(\tau)$ and $I_{H(G)}(P,\tau,Q)$ is simple.  Then  $P(\cV)=P(\cV')$ for any non-zero $H(M)_{ R}$-submodule  $\cV'$ of $\cV$;  the map 
 $\cV'\mapsto I_{H(G)}(P,\cV',Q): \mathcal L_{\cV}\to  \mathcal L_{I_{H(G)}(P,\cV,Q})$   is a  lattice isomorphism.\end{maprop}

\begin{proof} $P(\cV)=P(\cV')$ is proved as  Proposition \ref{prop:tensorstG}. We are in the situation of  
 Theorem \ref{thm:latticeA2} when $A=H(M'_\cV)_R\simeq \theta(H(M'_\cV)_R) \subset A'=H(M(\sigma))_R$ with the $*$-basis,  the $R$-representations $\St_Q^{H(M(\cV))}(R)$ and $e(\cV)$ of $M(\cV)$.  So   $\St_Q^{H(M(\cV))}(\cV)$  has finite length, and 
 its irreducible subquotients  are 
 $ \St_Q^{H(M(\cV))}(\tau)$ for the irreducible subquotients $\tau$ of $\cV$.    If 
 $I_G(P, \tau,Q)= \Ind_{P(\cV)}^G ( \St_Q^{M(\cV)}(\tau))$ is  irreducible for all $\tau$, we are in the situation of   Theorem \ref{thm:latticeA} for $F=\Ind_{P(\cV)}^{H(G)}$ and $W= \St_Q^{H(M(\cV))}(\cV)$ because 
 $\Ind_{P(\cV)}^{H(G)} $  has a right adjoint and  is exact  fully faithful (Proposition \ref{prop:RI} (iii)) so the map
 $\cV'\mapsto I_G(P,\cV',Q): \mathcal L_{\cV}\to  \mathcal L_{I_G(P,\cV,Q)}$   is a  lattice isomorphism.
 \end{proof}

  We  check  now that  the compatibility of $I_{H(G)}(P,\cV,Q)$   with  scalar extension, as for $I_{ G}(P,\sigma,Q)$ (Proposition \ref{prop:es2}).

   \begin{maprop}\label{prop:scalarLH} (i)   Let $R'/R$ be a field extension. The parabolic induction commutes with the scalar restriction  from $R'$ to $R$ and with the scalar extension from $R$ to $R'$. Hence
  the  left (resp. right) adjoint of the parabolic induction commutes with scalar extension (resp. restriction). This is valid for any commutative rings $R\subset R'$.

(ii)  Let $R'/R$ be a field extension. For  an  $H(M)_{R'}$-module  $\cV'$ 
 and an  $H(G)_{R}$-module $\cx$  such that $\Ind_P ^{H(G)} \cV'\simeq \cx_{R'}$, we have  $\cV'  \simeq (L_P^{H(G)}\cx)_{R'} $.
  
(ii)  Let $R'/R$ be an extension. For an $R$-triple  $(P,\cV,Q)$ of $H(G)$ we have:
  
$P(\cV)= P(\cV_{R'})$;  if $\cV$ is simple and $\cV'$ a subquotient   of $ \cV_{R'}$, then $P(\cV)=P(\cV')$.
   
  $  (e(\cV))_{R'}= e(\cV_{R'})  , \ \St_Q^{P(\cV)}(\cV)_{R'}\simeq  \St_Q^{P(\cV)}(\cV_{R'})$, and  $I_{H(G)}(P,\cV,Q)_{R'}\simeq I_{H(G)}(P,\cV_{R'},Q)$.

 (iii)  Let $R'$ be a subfield of $R$  and $(P,\cV,Q)$   an $R$-triple of $H(G)$ such that 
 $e(\cV)$, resp.  $  \St_Q^{H(M(\cV))}(\cV)$, resp. $I_{H(G)}(P,\cV,Q)$,   is the scalar extension of a $H(M(\cV))_{R'}$-module $\tau$, resp. $\rho$, resp. a $H(G)_{R'}$-module $\pi$.  
 
 Then  $\cV$ is the scalar extension to $R$ of the natural  action of  $H(M)_{R'}$ on 
$\tau $, resp. $ \St_{Q}^{P(\cV)} (\rho) $,  resp. $ \Hom_{H(M'_{\cV})_R}( \St_Q^{P(\cV)}(R),L_{P(\sigma)}^{H(G)}\pi) $.

\end{maprop}

\begin{proof}  (i)  As for the group (Prop. \ref{prop:scalarL}). it is clear that  the parabolic induction \eqref{eq:indH} commutes with the scalar restriction  from $R$ to $R'$ and with the scalar extension from $R'$ to $R$. Hence
  the  left (resp. right) adjoint of the parabolic induction commutes with scalar extension (resp. restriction).

(ii)  Let  $\cV$ be an  $H(M)_R$-module. As  an $H(M)_R$-module,  $\cV_{R'}  $ is  $\cV$-isotypic and this holds true for any subquotient  $\cV'$ of  $\cV_{R'}$  if $\cV$ is simple. For   $\alpha\in \Delta$ orthogonal to $\Delta_P$,  $T^{M,*}(z)$ for $z\in Z\cap M_\alpha'$ acts trivially on  $H(M)_{R'}$-module  if and only if   it acts trivially on this module seen as an $H(M)_{R}$-module.  So $P(\cV)= P(\cV_{R'})$ and  if $\cV$ is simple $P(\cV)=P(\cV')$. It is clear that   $(e(\cV))_{R'}\simeq  e(\cV_{R'})$. As
$$e(\cV) \otimes_R\St_Q^{H(M(\cV) )}(R)\simeq  e(\cV) \otimes_{\mathbb Z} \St_Q^{H(M(\cV)  )}\simeq    \St_Q^{H(M(\cV)  )}(\cV),$$
we have $\St_Q^{H(M(\cV)  )}(\cV)_{R'}\simeq \St_Q^{H(M(\cV)  )}(\cV_{R'})$. As $\Ind_P^{H(G)}$ commutes with scalar extension (i),  we have $I_{H(G)}(P,\cV,Q)_{R'}=(\Ind_{P(\cV)}^{H(G)}(St_Q^{H(M(\cV)  )}(\cV)))_{R'} \simeq I_{H(G)}(P,\cV_{R'},Q)$.

(iii) If $I_{H(G)}(P,\cV,Q)=\pi_R$ then $\St_{Q}^{H(M (\cV))} (\cV)  =\rho_R$  where 
$\rho \simeq L_{H(M(\cV))}^{H(G)}(\pi)$ by (i) and \eqref{eq:Lst}.

If $\St_{Q}^{H(M (\cV))} (\cV)  =\rho_R$, then  $e(\cV)=\tau_R$ where $ \tau\simeq\Hom_{H(M'_{\cV})_{R'} }(\St_{Q}^{H(M (\cV))}(R'), \rho)$ as   $e(\cV)\simeq   \Hom_{H(M'_{\cV})_{R} }(\St_{Q}^{H(M (\cV))}(R),\rho_R) $ (Prop. \ref{prop:tensorstHG}) and $\St_{Q}^{H(M (\cV))}(R)\simeq (\St_{Q}^{H(M (\cV))}(R'))_R$.
   
 If $e(\cV)= \tau_{R}$ then $T^{M(\cV),*}(m)$ acts trivially on $\tau_{R}$ for $m\in M'_{\cV}$ hence  also on $\tau$ and $\cV=\cz_R$ where $\cz$ is the restriction of $\tau$ to $H(M)_R$.
\end{proof}

\begin{remark}\label{rem:tensorstHG}{\rm   Proposition \ref{prop:tensorstH} applies to  the scalar extension $\cV _{R^{alg}}$  to $R^{alg}$ of  a simple supersingular $H(M)_R$-module $\cV$; the proof is the same as for the group (Remark \ref{rem:tensorstG}). By  the decomposition theorem of  $\cV _{R^{alg}}$ and  Lemma \ref{lem:ss} all irreducible subquotients $\tau$ of  $\cV _{R^{alg}}$ are supersingular,  $P(\tau)=P(\cV)=  P(\cV_{R^{alg}})$  (Prop.\ref{prop:scalarLH} (ii)), and $I_{H(G)}(P,\tau,Q)$ is irreducible  by the classification theorem for $H(G)$  over  $ R^{alg}$ (Thm.\ref{thm:classHG} \cite{AHenV2}). }
\end{remark}

  \subsection{Classification of simple modules over the pro-$p$ Iwahori Hecke algebra} \label{S:6}   
As in  \S\ref{S:3.4} for $G$   the classification theorem for $H(G)$  over  $R^{alg}$   (Thm.\ref{thm:classHG}) descends to $R$  by a formal proof  relying on  the  properties:
    
  {\bf 1}  The decomposition theorem  for $H(G)$   (Thm.\ref{thm:DA}).         

{\bf 2}     The classification theorem for $H(G)$ over  $ R^{alg}$ (Thm.\ref{thm:classHG} \cite{AHenV2}).

{\bf 3}  The compatibility of   scalar extension with  $I_{H(G)}(P, -, Q)$ (Prop. \ref{prop:scalarLH}) and  supersingularity   (Lemma  \ref{lem:ss}).

 {\bf 4}  The lattice isomorphism $\mathcal L_{\cV_{R^{alg}}}\to \mathcal L_{I_{H(G)}(P,\cV_{R^{alg}},Q)}$  for  the scalar extension $\cV_{R^{alg}}$ to $R^{alg}$ of a simple supersingular  $H(M)_R$-module  $\cV$   (Prop.\ref{prop:tensorstHG}  and  Remark  \ref{rem:tensorstHG}).

\bigskip  We start the proof with an arbitrary simple  $H(G)_R$-module  $\cx$. By the decomposition theorem,    the $H(G)_{R^{alg}}$-module $\cx_{R^{alg}} $   has finite length;   we choose a simple submodule $\cx^{alg}$ of $\cx_{R^{alg}} $.  By the classification theorem  over $R^{alg}$, 
\begin{equation}\label{eq:cx} \cx^{alg}  \simeq I_{H(G)}(P, \cV^{alg},Q)
\end{equation}
for an $R^{alg}$-triple
$(P=MN, \cV^{alg},Q)$  of $G$ where   $\cV^{alg} $   is a simple  supersingular $H(M)_{R^{alg}}$-module. By the decomposition theorem, $\cx^{alg}$ descends to a finite extension of $R$, and also $\cV^{alg} $ by compatibility of   scalar extension with  $I_{H(G)}(P, -, Q)$. By the decomposition theorem,
$ \cV^{alg} $ is contained in   the scalar extension $\cV_{R^{alg}}$  to   $R^{alg}$
 of a simple  $H(M)_R$-module  $\cV$. By  compatibility   of   scalar extension with  $I_{H(G)}(P, -, Q)$ and supersingularity,  $\cV $  is supersingular, $(P,\cV, Q)$ is an $R$-triple of $G $  and 
 \begin{equation}\label{eq:cv} I_{H(G)}(P,\cV_{R^{alg}},Q)\simeq I_{H(G)}(P,\cV,Q)_{R^{alg}}.
\end{equation} We have $I_{H(G)}(P,\cV^{alg},Q) \subset I_{H(G)}(P,\cV_{R^{alg}},Q)$ by the lattice  isomorphism  $\mathcal L_{\cV _{R^{alg}}}\to \mathcal L_{I_{H(G)}(P,\cV _{R^{alg}},Q)}$.    
The decomposition theorem implies  
 $$\cx\simeq I_{H(G)}(P,\cV,Q).$$

Conversely,
we  start with an  
 $R$-triple $(P, \cV, Q)$  of $H(G)$ with  $\cV $  simple  supersingular and we prove that $I_{H(G)}(P,\cV, Q)$ is simple.    By the decomposition theorem,  the $H(G)_{R^{alg}}$-module  $\cV_{R^{alg}}$ has finite length, and   $ I_{H(G)}(P,\cV_{R^{alg}}, Q)$ also  by the lattice  isomorphism  $\mathcal L_{\cV _{R^{alg}}}\to \mathcal L_{I_{H(G)}(P,\cV _{R^{alg}},Q)}$. The scalar extension is faithful and exact and $I_{H(G)}(P,\cV, Q)_{R^{alg}}  \simeq I_{H(G)}(P,\cV_{R^{alg}}, Q)$  so $I_{H(G)}(P,\cV, Q)$ has also  finite length. 
 We choose  a simple $H(G)_R$-submodule $\cx$  of $I_{H(G)}(P,\cV, Q)$.  The  $H(G)_{R^{alg}}$-module  $\cx_{R^{alg}}$ is contained in $I_{H(G)}(P,\cV, Q)_{R^{alg}}$ hence  $\cx_{R^{alg}}\simeq I_{H(G)}P,\cV', Q) $ for an $H(M)_{R^{alg}}$-submodule $\cV' $ of $\cV_{R^{alg}}$ by (\ref{eq:cv}) and the lattice  isomorphism  $\mathcal L_{\cV _{R^{alg}}}\to \mathcal L_{I_{H(G)}(P,\cV _{R^{alg}},Q)}$. As $I_{H(G)}(P,\cV', Q) $ descends to $R$, $\cV'$ is also. But  no proper $H(M)_{R^{alg}}$-submodule of $\cV_{R^{alg}}$  descends to $R$ by the decomposition theorem for $H(G)$,  so $\cV'=\cV_{R^{alg}}$,  $\cx_{R^{alg}}= I_{H(G)}(P,\cV_{R^{alg}}, Q)$ and $\cx_{R^{alg}}\simeq I_{H(G)}(P,\cV, Q)_{R^{alg}}$  by compatibility of   scalar extension with  $I_{H(G)}(P, -, Q)$.    So $ \cx\simeq I_{H(G)}(P,\cV, Q) $  and $I_{H(G)}(P,\cV, Q)$ is simple.

Finally, let  $(P, \cV, Q)$ and  $(P_1, \cV_1, Q_1)$ be  two  $R$-triples  of ${H(G)}$ with  $\cV, \cV_1 $ simple supersingular and $ I_{H(G)}(P,\cV, Q) \simeq  I_{H(G)}(P_1,\cV_1, Q_1)$. The scalar extensions to $R^{alg}$ are isomorphic 
  $( I_{H(G)}(P,\cV, Q))_{R^{alg}} \simeq ( I_{H(G)}(P_1,(\cV_1), Q_1))_{R^{alg}} $. The   classification theorem for $H(G)$ over $R^{alg}$ and (\ref{eq:cv})  imply $P=P_1, Q=Q_1$ and some simple $H(M)_{R^{alg}}$-subquotient $\cV^{alg}$ of $\cV_{R^{alg}}$ is isomorphic to some  simple $H(M)_{R^{alg}}$-subquotient $\cV_1^{alg}$ of $(\cV_1)_{R^{alg}}$.  As  $\cV^{alg}$ is $\cV$-isotypic and $\cV_1^{alg}$ is  $\cV_1$-isotypic as $H(M)_R$-module,  $\cV$ and $ \cV_1 $ are isomorphic.
 
 This ends the proof of the  classification theorem for $H(G)$ (Thm.\ref{thm:classHG}). $\square$
 
 \section{Applications}\label{S:5}
 
 Let $R$ be a field of characteristic $p$ and $G$ a reductive $p$-adic group as in \S\ref{S:II.1}.

\subsection{Vanishing of the smooth dual}\label{S:3.5} 

The dual of   $\pi\in \Mod_R(G)$  is $ \Hom_R(\pi,R)$ with the contragredient action of $G$, that is,  $(gf)(gx)=f(x)$ for $g\in G, f\in \Hom_R(\pi,R), x\in \pi$. 
 The smooth dual of $\pi$ 
   is $\pi^\vee := \cup_K\Hom_R(\pi,R)^K$ where $K$ runs through the open compact subgroups of $G$.

A  finite dimensional  smooth  $R$-representation of $G$ is fixed by an open compact subgroup, and its smooth dual  is equal to its dual.

\bigskip We prove Thm. \ref{thm:vsd}.
 Let $R^{alg}/R$ be an algebraic closure and let $\pi$ be a non-zero irreducible admissible $R$-representation $\pi$ of $G$. By Remark \ref{rem:2}, $(\pi^\vee)_{R^{alg}}\subset (\pi_{R^{alg}})^\vee$.  Assume that $\pi^\vee\neq 0$. Then,  $(\pi^\vee)_{R^{alg}}\neq 0$, hence  $(\pi_{R^{alg}})^\vee \neq 0$ implying  $\rho^\vee \neq 0$  for some irreducible subquotient $\rho$ of $\pi_{R^{alg}}$. By the theorem over $R^{alg}$ \cite[Thm.6.4]{AHenV2}, the $R^{alg}$-dimension of $\rho$ is finite. The $R^{alg}$-dimension is constant on the $\Aut_R(R^{alg})$-orbit  of $\rho$. By the decomposition theorem (Thm. \ref{thm:DG}), the $R^{alg}$-dimension 
 of $\pi_{R^{alg}}$ is finite. It is equal to the $R$-dimension of $\pi$. So we proved that $\pi^\vee\neq 0$ implies that the $R$-dimension of $\pi$ is finite. $\square$

\subsection{Lattice of submodules {\rm (Proof of Theorem \ref{thm:lattice})}}\label{S:54}

\

\subsubsection{}\label{ss:sub1} Before proving Thm. \ref{thm:lattice}, we recall   some properties  of the invariant functor $(-)^I:\Mod_R^\infty(G)\to \Mod_R (H(G))$ (see \S\ref{s:ira}) which are specific to  the pro-$p$ Iwahori group $I$. Denote by $\sigma$ 
a smooth $R$-representation of $M$ and  by  $\cV$ a   $H(M)_R$-module.

1.  The parabolic induction commutes with $(-)^I$ \cite[Prop.4.4]{OV} and with  $- \otimes_{H(G)_R}R[I\backslash G]$ \cite[Cor.4.7]{OV}: $$(\Ind_P^G \sigma)^I\simeq \Ind_P^{H(G)} (\sigma^{I\cap M}), \ 
 \Ind_P^{H(G)}\cV \otimes_{H(G)_R}R[I\backslash G]\simeq \Ind_P^G (\cV \otimes_{H(M)_R}R[(I\cap M)\backslash M]).$$

2. The last isomorphism and the faithfulness  $\Ind_P^{H(G)}$ (Prop.\ref{prop:RI}) show that  the natural map  
 $$\eta_{\Ind_P^{H(G)}\cV}: \Ind_P^{H(G)}\cV \to  (\Ind_P^{H(G)}\cV \otimes_{H(G)}\mathbb Z[I\backslash G] )^I$$ 
is bijective if and only if   the natural map $\eta_{\cV}: \cV \to (\cV \otimes_{H(M)}\mathbb Z[(I\cap M)\backslash M])^{I\cap M}$ is.

3. The trivial $R$-representation of $G$ is naturally isomorphic to $R\otimes_{H(G)}\mathbb Z[I\backslash G])^{I}$  \cite[Lemma 2.25]{OV}.
  
4. The $I$-invariants of $I_G(P,\sigma,Q)$ is  isomorphic to $I_G(P,\sigma^{I\cap M},Q)$  when $\sigma=\sigma_{min}$  (\S\ref{S:II.1}) and $P(\sigma)=P(\sigma^{ I\cap M})$.

 \begin{monlem}\label{lem:sub1} Let $\sigma$ be an irreducible admissible supersingular $R$-representation of $M$. Then $\sigma=\sigma_{min}$, $P(\sigma)=P(\sigma^{I\cap M})$, so $I_G(P,\sigma,Q)^I\simeq I_ {H(G)}(P, \sigma^{I\cap M},Q)$.
\end{monlem}
\begin{proof}  The equality $\sigma=\sigma_{min}$  follows from the classification (Thm.\ref{thm:classG}) because $\sigma$ is supersingular {\rm (\S\ref{S:ss})}. When $\sigma=\sigma_{min}$, then $\Delta_\sigma$ is orthogonal to $\Delta_M$ (\S\ref{S:II.1}).  As $\sigma$ being  irreducible is generated by $ \sigma^{I\cap M}$, $P(\sigma)=P(\sigma^{I\cap M})$  \cite[Thm.3.13]{AHenV2}.
\end{proof}

5. The   representations $ I_{H(G)}(P, \cV ,Q)\otimes_{H(G)_R}R[I\backslash G]$ and $ I_{G}(P, \cV \otimes_{H(M)_R}R[(I\cap M)\backslash M] , Q)$ 
are isomorphic when  $\cV$ is simple and supersingular (when $\cV \otimes_{H(M)_R}R[(I\cap M)\backslash M]= 0$ or   $P(\cV)=P(\cV \otimes_{H(M)_R}R[(I\cap M)\backslash M])$ when it is not $0$, more generally)  \cite[Cor. 5.12, 5.13]{AHenV2}.

\ \subsubsection{$\Ind_P^G(R)$ and  $\Ind_P^{H(G)}(R)$}\label{ss:sub2}

It is known that  $\Ind_P^G(R)$  is multiplicity free of irreducible subquotients $\St_Q^G(R)$ and  $\Ind_Q^G (R) $  is the   subrepresentation of  $\Ind_P^G (R)$ with cosocle $\St_Q^G(R)$, for  $P\subset Q\subset G$ \cite[\S9]{Ly}.  

Therefore,  sending  $\St_Q^G(R) $ for $P\subset Q$ to $\Delta_Q\setminus \Delta_P$ induces a lattice  isomorphism from $\mathcal L_{\Ind_P^G(R) }$  onto the set of upper sets in  $\mathcal P(\Delta \setminus \Delta_P)$;  to an upper set in $\mathcal P(\Delta \setminus \Delta_P)$ is associated the subrepresentation  $\sum_J\Ind_{P_{J \cup \Delta_P} }^G (R) $  for $J$  in the  upper set  \cite[Prop.3.6]{AHenV2}. 

The  $H(G)_R$-module
$\Ind_P^{H(G)}(R)$ has a filtration with quotients  $\St_Q^{H(G)}(R)$ for  $P\subset Q\subset G$. By the classification theorem, the $\St_Q^{H(G)}(R)$ are simple not isomorphic. So $\Ind_P^{H(G)}(R)$ is multiplicity free of simple subquotients  $\St_Q^{H(G)}(R)$  for  $P\subset Q\subset G$.    

Applying 1, 2 and 3 in \S\ref{S:54},  the natural map $$\Ind_P^{H(G)}(R)\otimes_{H(G)}\mathbb Z[I \backslash G]\to \Ind_P^G (R)$$
is an isomorphism and $\eta_{\Ind_P^{H(G)}(R)}$ is bijective.

 The properties a), b'), c') of Theorem \ref{thm:latticeA} are satisfied for  the functor $-\otimes_{H(G)}\mathbb Z[I\backslash G]:\Mod_R(H(G))\to \Mod_R(G)$ of right adjoint $(-)^I$, and the $H(G)_R$-module  $\Ind_P^{H(G)}(R)$. 
 So $(-\otimes_{H(G)}\mathbb Z[I\backslash G], (-)^I)$ give lattice isomorphisms between $\mathcal L_{\Ind_P^{H(G)}(R)}$ and $\mathcal L_{\Ind_P^G(R) }$.

\subsubsection{$\Ind_P^{G}(\St_{Q}^{M}(R))$ and $\Ind_P^{H(G)}( \St_Q^{H(M)}(R))$ for $Q\subset P$}\label{ss:sub3}   This  case is a direct consequence of  \S\ref{ss:sub2}   because 
$$\Ind_P^{G}(\St_{Q}^{M}(R))= \Ind_Q^{G} (R)/ \sum_{Q\subsetneq Q_1 \subset P} \Ind_{Q_1}^{G}  (R) $$  is a   quotient of  $\Ind_Q^{G}  (R)$. We deduce from  \S\ref{ss:sub2} that 
 $\Ind_P^{G}(\St_{Q}^{M}(R))$ is multiplicity free of irreducible subquotients  $\St_{Q'}^G(R)$ for $Q\subset Q'$ but  $Q'$ does not contain any $Q_1$ such that $Q\subsetneq Q_1 \subset P$, that is, $ Q = Q'\cap P$.  The  subrepresentation  $\Ind_P^{G}(\St_{Q'}^{M}(R))$ of $\Ind_P^{G}(\St_{Q}^{M}(R))$ has cosocle  $\St_{Q'}^G$.
Sending $\St_{Q'}^{G}(R)$  to $\Delta_{Q'}\cap (\Delta \setminus \Delta_P)$ gives a lattice isomorphism  from  $\mathcal L_{\Ind_P^{G}(\St_{Q}^{M}(R))}$ onto the lattice of upper sets in $\mathcal P (\Delta\setminus \Delta_P) $ (which does not depend on $Q$).  
We deduce also from \S\ref{ss:sub2} and Remark \ref{rem:latticesq} that $-\otimes_{H(G)}\mathbb Z[I\backslash G]$ and $ (-)^I$ give lattice isomorphisms between $\mathcal L_{\Ind_P^{H(G)}( \St_Q^{H(M)}(R))}$ and $\mathcal L_{\Ind_P^{G}(\St_{Q}^{M}(R)) }$.

\subsubsection{$\Ind_P^G\sigma$  for $\sigma$ irreducible admissible supersingular and   $\Ind_P^{H(G )}\cV $ for $\cV$ simple supersingular}\label{ss:sub4}   
 $\Ind_P^G\sigma $ admits a filtration with quotients $I_{G}(P,\sigma,Q)= \Ind_{P(\sigma)}^{G}(   \St_Q^{M(\sigma)}(\sigma))$ for $P\subset Q\subset P(\sigma)$ and by the classification theorem, the  $I_{G}(P,\sigma,Q)$ are irreducible  and not isomorphic; so $\Ind_P^G(\sigma )$ is multiplicity free of irreducible subquotients  $I_{G}(P,\sigma,Q)$ for the $R$-triples $(  P, \sigma, Q )$ of $G$. 
The  maps 
$$X\mapsto  e(\sigma)\otimes_R X\mapsto \Ind_{P(\sigma)}^G(e(\sigma)\otimes_R X):\mathcal L_{\Ind_P^{M(\sigma)}(R)}\to  \mathcal L_{e(\sigma)\otimes_R  \Ind_P^{M(\sigma)}(R)}   \to \mathcal L_{\Ind_P^G(\sigma)}$$ are lattice isomorphisms: this follows from the lattice theorems   and the classification theorem (Thm.\ref{thm:latticeA}, Thm.\ref{thm:latticeA2}, Thm.\ref{thm:classG}), as   in Proposition \ref{prop:tensorstG} (When $R$ is algebraically closed \cite[Prop.3.8]{AHenV1}). 

  For a simple supersingular $H(M)_R$-module $\cV$, the same arguments show that $\Ind_P^{H(G)}(\cV )$ is multiplicity free of simple subquotients  $I_{H(G)}(P,\cV,Q)$ for $  P\subset Q\subset P(\cV)$. The  maps 
$$Y\mapsto  e(\cV)\otimes_R Y\mapsto \Ind_{P(\cV)}^{H(G)}(e(\cV)\otimes_R Y):\mathcal L_{\Ind_P^{H(M(\cV))}(R)}\to  \mathcal L_{e(\cV)\otimes_R  \Ind_P^{H(M(\cV)}(R))}   \to \mathcal L_{\Ind_P^{H(G)}(\cV)}$$ are lattice isomorphisms, by applying Thm.\ref{thm:latticeA}, Thm.\ref{thm:latticeA2}, Thm.\ref{thm:classHG}, as in Proposition \ref{prop:tensorstH}.

\subsubsection{$\Ind_P^G ( \St_{Q}^{M}(\sigma_1))$   and 
$\Ind_P^{H(G)} ( \St_{Q}^{H(M)}(\cV_1))$ for an $R$-triple $(P_1 ,\sigma_1, P)$ of $G$,  $P_1 \subset Q\subset P $, $\sigma_1$ irreducible admissible supersingular and similarly for $\cV$} \label{ss:sub5} This is a direct consequence of \S\ref{ss:sub4}   because  $$\Ind_P^G ( \St_{Q}^{M}(\sigma_1)) =(\Ind_{Q}^Ge_{Q}(\sigma_1))/(\sum_{Q \subsetneq Q_1 \subset P}\Ind_{Q_1}^G e_{Q_1}(\sigma_1))$$
  is a subquotient of $\Ind_{P_1}^{G}(\sigma_1)$ as   $e_{Q}(\sigma_1) \subset \Ind_{P_1}^{M_Q}(\sigma_1)$ and similarly for $\cV$.
 We have 
$\Ind_{Q_1}^G e_{Q_1}(\sigma_1)\simeq  \Ind_{P(\sigma_1)}^G(e(\sigma_1)\otimes_R \ind_{Q_1}^{M(\sigma_1)}(R))$,  
 and a  lattice isomorphism  (\S\ref{ss:sub4}):
    $$X\mapsto \Ind_{P(\sigma_1)}^{G}(e(\sigma_1)\otimes_R X):\mathcal L_{\Ind_{P_1}^{M(\sigma_1)}(R)}\to \mathcal L_{\Ind_{P_1}^{G}(\sigma_1)}$$
    inducing a 
    lattice isomorphism  (Remark \ref{rem:latticesq}):
     $$\mathcal L_{\Ind_{P }^{M(\sigma_1)}(\St_Q^{M }(R))}\to \mathcal L_{\Ind_P^G ( \St_{Q}^{M}(\sigma_1)) } .$$
  The $R$-representation  $\Ind_P^G ( \St_{Q}^{M}(\sigma_1))$ is multiplicity free of irreducible subquotients $I_G(P_1,\sigma_1, Q')$ for the $R$-triples $(P_1,\sigma_1, Q')$ of $G$ with $Q'\cap P=Q$ (\S\ref{ss:sub3}).     And similarly for $\cV$ with the same arguments and references.

 \subsubsection{  $\Ind_P^G\sigma$ for $\sigma$  irreducible admissible and $\Ind_P^{H(G )}\cV $ for $\cV$ simple} \label{ss:sub6}  By the classification theorem, there exists an $R$-triple $(P_1,\sigma_1, Q)$ of $G$ with $Q\subset P$, $\sigma_1$ irreducible admissible supersingular such that 
$$\sigma\simeq I_M(P_1\cap M, \sigma_1, Q\cap M)= \Ind_{P(\sigma_1)\cap M }^M(  \St_{Q\cap M}^{M(\sigma_1)\cap M}(\sigma_1)).$$
The transitivity of the induction   implies  $\Ind_P^G \sigma \simeq  \Ind_{P(\sigma_1)\cap P}^G(  \St_{Q}^{M(\sigma_1)\cap M}(\sigma_1))$. This is  the  case 
 \S\ref{ss:sub5} with $P(\sigma_1)\cap P$. The $R$-representation $\Ind_P^G\sigma$ of $G$ is multiciplicity free of irreducible subquotients $I_G(P_1, \sigma_1, Q')$ for the $R$-triples $(P_1, \sigma_1, Q')$ of $G$ with $Q'\cap P=Q$ (note that $Q'\subset P(\sigma_1), Q \subset P$).
 The map
  \begin{align*}X\mapsto    \Ind_{P(\sigma_1) }^G(e(\sigma_1)\otimes_R X): & \ \mathcal L_{\Ind_{P(\sigma_1)\cap P}^{M(\sigma_1)}(\St_{Q}^{M(\sigma_1)\cap M}(R))}    \to \mathcal L_{\Ind_{P}^G (\sigma)} \end{align*}  
 is a lattice isomorphism.      And similarly for $\cV$ with the same arguments and references.

\subsubsection{Invariants by the pro-$p$ Iwahori} \label{ss:sub7}  We keep the notations of \S\ref{ss:sub6}. 
   The classification theorem  shows that  
$$ \sigma^{I\cap M} \ \text{ is simple } \ \Leftrightarrow \  \sigma_1^{I\cap M_1} \  \text{ is simple }$$
because  $\sigma^{I\cap M}\simeq I_{H(M)}(P_1\cap M, \sigma_1^{I\cap M_1}, Q\cap M)$ (\S\ref{ss:sub1}) and  $\sigma_1^{I\cap M_1}$ is supersingular of finite length. 
   
\medskip    Assume  first  that $P(\sigma_1)=P(\cV_1)$ in \S\ref{ss:sub6}.  In \S\ref{ss:sub3} we saw that the  maps \begin{equation}\label{eq:sub6}X\mapsto X^{I\cap M(\sigma_1)}, \quad Y\mapsto Y\otimes_{H(M(\sigma_1))} \mathbb Z[I\cap M(\sigma_1)\backslash M(\sigma_1)]
\end{equation} between  $\mathcal L_{\Ind_{P(\sigma_1)\cap P}^{M(\sigma_1)}(\St_{Q}^{M(\sigma_1)\cap M}(R))}  $ and $ \mathcal L_{\Ind_{P(\sigma_1)\cap P}^{H(M(\sigma_1))}(\St_{Q}^{H(M(\sigma_1)\cap M)}(R))}$, 
 are
   lattice isomorphisms, inverse from each other.     
They induce   lattice isomorphisms,  inverse of each other,  between   $\mathcal L_{\Ind_{P}^G (\sigma)}$ and    $\mathcal L_{\Ind_P^{H(G )}(\cV )}$: 
  \begin{align} \label{eq:l1} \Ind_{P(\sigma_1) }^G(e(\sigma_1)\otimes_R X) &\mapsto  \Ind_{P(\cV_1) }^{H(G )}(e(\cV_1)\otimes_R X^{I\cap M(\sigma_1)}), \\
  \label{eq:l2}\Ind_{P(\cV_1) }^{H(G )}(e(\cV_1)\otimes_R Y) &\mapsto    \Ind_{P(\sigma_1) }^G(e(\sigma_1)\otimes_R (Y\otimes_{H(M(\sigma_1))} \mathbb Z[I\cap M(\sigma_1)\backslash M(\sigma_1)])).
\end{align}
by the lattice isomorphisms of  \S\ref{ss:sub6}  with $\mathcal L_{\Ind_{P}^G (\sigma)}$ and    $\mathcal L_{\Ind_P^{H(G )}(\cV )}$

Assume  now (until the end of \S\ref{ss:sub7}) that $\sigma^{I\cap M} $ and  $\sigma_1^{I\cap M_1} $  are simple, and that the natural map $\sigma^{I\cap M}\otimes_{H(M)}\mathbb Z[(I\cap M)\backslash M] \to \sigma$ is injective.

Then $P(\sigma_1)=P(\sigma_1^{I\cap M_1})$ (Lemma  \ref{lem:sub1}). When $\cV_1= \sigma_1^{I\cap M_1}$,  the lattice isomorphisms \eqref{eq:l1}, \eqref{eq:l2}   between $\mathcal L_{\Ind_P^G (  \sigma)}$ and $  \mathcal L_{\Ind_P^{H(G )}(  \sigma^{I\cap M}) }$ are simply given by $(-)^I$ and  $-\otimes_{H(M(G)} \mathbb Z[I \backslash G]$:
\begin{align}\label{eq:sub7}\Ind_{P(\sigma_1) }^G(e(\sigma_1)\otimes_R X)&\mapsto (\Ind_{P(\sigma_1) }^G(e(\sigma_1)\otimes_R X))^{I\ },\\
\Ind_{P(\sigma_1) }^{H(G )}(e(\sigma_1^{I\cap M_1})\otimes_R Y) & \mapsto (\Ind_{P(\sigma_1) }^{H(G )}(e(\sigma_1^{I\cap M_1})\otimes_R Y))\otimes_{H(M(G)} \mathbb Z[I \backslash G].
\end{align}
This follows from the lattice theorem (Thm.\ref{thm:latticeA}) applied to the functor $- \otimes_{H(G)}\mathbb Z[I\backslash G]:\Mod_R(H(G))\to \Mod_R(G)$ of right adjoint $(-)^I$ and to $ \Ind_P^{H(G )}(\sigma^{I\cap M} )$ after having checked that they satisfy the  hypotheses a), b'), c') of this theorem. 

As the map $\sigma^{I\cap M}\otimes_{H(M)}\mathbb Z[(I\cap M)\backslash M] \to \sigma$ is injective, $\sigma$ irreducible, and as the parabolic induction commutes with  $- \otimes_{H(G)}\mathbb Z[I\backslash G]$, the natural map 
$$\Ind_P^{H(G )}(\sigma^{I\cap M} )\otimes_{H(G)}\mathbb Z[I\backslash G])\to \Ind_P^G\sigma$$ is an isomorphism. The commutativity of the parabolic induction with $(-)^I $ (\S\ref{ss:sub1}) implies that  $ \Ind_P^{H(G )}(\sigma^{I\cap M} )\to ( \Ind_P^{H(G )}(\sigma^{I\cap M} )\otimes_{H(G)}\mathbb Z[I\backslash G])^I$ is an isomorphism, i.e. a).
 The $H(G)_R$-module $ \Ind_P^{H(G )}(\sigma^{I\cap M} )$ and $\Ind_P^G\sigma$ have finite length (\S\ref{ss:sub6}), i.e. c'). 
 For b'), i.e. $\pi^I$ is simple for any
 irreducible subquotient $\pi$ of $ \Ind_Q^G (  \sigma)$,    we write   $\pi\simeq  I_G(P_1,\sigma_1,Q')$ for an $R$-triple $(P_1,\sigma_1,Q')$ of $G$ with $Q'\cap P=Q$. By lemma \ref{lem:sub1},  $I_G(P_1,\sigma_1,Q')^I \simeq I_{H(G)}(P_1,\sigma_1^{I\cap M_1},Q')$  and $I_{H(G)}(P_1,\sigma_1^{I\cap M_1},Q')$ is simple by the classification theorem.

This ends the proof of Theorem \ref{thm:lattice}.

\section{Appendix: Eight inductions   $\Mod_R(H(M))\to \Mod_R(H(G))$} \label{s:ind}
 
For  a commutative ring $R$ and a parabolic subgroup $P=MN$ of $G$, there are  eight different   inductions 
 $\Mod_R(H(M))\to \Mod_R(H(G))$
$$-        \otimes_{H(M^\epsilon),\theta^{\eta}}H(G)  \quad and \quad  \Hom_{H(M^\epsilon),\theta^{\eta}}(H(G), -)\quad {for}\ \epsilon\in \{+,-\}, \ \eta\in \{ \ , *\}.$$
  associated to  the eight    elements  of  $\{\otimes, \Hom\} \times \{+,-\}\times \{\theta, \theta^*\}$ where  $\theta:=\theta_M^G$ \ref{s:ring}. We  We write   $\{ \theta^\eta, \theta^{*\eta}\}=\{\theta, \theta^*\}$ (as sets).
The triple $(\otimes ,+, \theta)$ corresponds to 
    the parabolic induction $\Ind_P^{H(G)}(-)=-\otimes_{H(M^+),\theta} H(G)$     and the triple $(\Hom ,-,\theta^*)$ corresponds to  $CI_P^{H(G)}(-) =\Hom_{H(M^-),\theta^*}(H(G), -)$  that we call parabolic coinduction.  The propositions (Prop.\ref{prop:exc}, Prop.\ref{prop:duality}) comparing these  eight inductions, are extracted from \cite{Abeparind} and \cite{Abeinv}. To formulate   them we   need first to define the `` twist by  $n_{w_Gw_M}$'' and the  involution $\iota^M_{\ell-\ell_M}$of $H(M)$.

\medskip 

  {\bf Twist by $n_{w_Gw_M}$.} Let  $w_M=w_P$ be  the longest element of the    Weyl group of $\Delta_M=\Delta_P$,  and  $w\mapsto n_w:\mathbb W \to W$ is an injective homomorphism from the Weyl group $\mathbb W$ of $\Delta$ to $W$ satisfying the braid relations (there is no canonical choice). 

Let $P^{op}=M^{op}N^{op}$ denote the parabolic subgroup of $G$ (containing $B$) with $\Delta_{M^{op}}=\Delta_{P^{op}}= w_Gw_P(\Delta_P)=w_G(-\Delta_P)$ (image of $\Delta_P$ by  the opposition involution  $\alpha\mapsto w_G(-\alpha)$ \cite[1.5.1]{T}).  The twist by $n_{w_Gw_M}$  is the  ring isomorphism  \cite[\S4.3]{Abe} $$H(M)\to H(M^{op}) \quad ( T^M_w,  T^{M,*}_w)\mapsto  (T^{M^{op}}_{n_{w_Gw_M} w n_{w_Gw_M}^{-1}}, T^{M^{op},*}_{n_{w_Gw_M} w n_{w_Gw_M}^{-1}}) \quad\text{ for} \  w\in W_M.  
$$
It restricts to an isomorphim  $H(M^\epsilon)\to H(M^{op,-\epsilon})$  (\cite[Prop.2.20]{VigpIwst}.
The inverse of the twist by $n_{w_Gw_M}$ is the twist by $n_{w_Gw_{M^{op}}}$, as  $n_{w_G w_{P^{op}}}=n_{w_P w_G}=n_{w_G w_P}^{-1}$.

By functoriality the twist by   by $n_{w_Gw_M}$ gives a functor 
$$\Mod_R(H(M))\xrightarrow{n_{w_Gw_M}(-)}
\Mod_R(H(M^{op})).$$

 \medskip 
  {\bf  Involution $\iota^M_{\ell-\ell_M}$}   \cite[\S 4.1]{Abeparind}.  The two commuting involutions $\iota^M$ and $\iota_{\ell-\ell_M}$  of the ring $H(M) $:

 $(T^M_w, T^{M,*}_w)\xrightarrow{\iota^M}  (-1)^{\ell_M(w)}(T^{M,*}_w , T^{M}_w)$   \cite[Prop. 4.23]{VigpIw},
 
  $ (T^M_w, T^{M,*}_w)\xrightarrow{\iota_{\ell-\ell_M}}  (-1)^{\ell(w)-\ell_M(w)} (T^M_w,   T^{M,*}_w)$   \cite[Lemmas 4.2, 4.3, 4.4, 4.5]{Abeparind}.
 
\noindent  give by   composition an involution $\iota^M_{\ell-\ell_M}$  of $H(M)$ 

$(T^M_w, T^{M,*}_w)\xrightarrow{\iota^M_{\ell-\ell_M}}(-1)^{\ell(w)}(T^{M,*}_w, T^{M}_w)$.

\noindent   The twist by $n_{w_Gw_M}$ and the involution $\iota^M_{\ell-\ell_M}$ commute: 
$$n_{w_Gw_M}(-) \circ \iota^M_{\ell-\ell_M} = \iota^{M^{op}}_{\ell-\ell_{M^{op}}} \circ n_{w_Gw_M}(-) :  H(M)\to H(M^{op}),$$   send  $T^M_w$ for $w\in W_M$ to
  $$ (-1)^{\ell(n_{w_Gw_M}wn_{w_Gw_M}^{-1})} T^{M^{op},*}_{n_{w_Gw_M}wn_{w_Gw_M}^{-1}}=
 (-1)^{\ell(w)}T^{M^{op},*}_{n_{w_Gw_M}wn_{w_Gw_M}^{-1}}$$
 (for the equality,  recall that the length $\ell_M$  of $W_M$  is invariant by conjugation by $w_M$, and
$ \ell(n_{w_Gw_M}wn_{w_Gw_M}^{-1})=\ell(n_{w_G} n_{w_M}wn_{w_M}^{-1}n_{w_G}^{-1})=\ell( n_{w_M}wn_{w_M}^{-1})=\ell( w )$). 
By functoriality we get   a functor 
$$\Mod_R(H(M))\xrightarrow{ (-)^{\iota^M_{\ell-\ell_M}}}
\Mod_R(H(M )).$$
 When $M=G$, we write simply  $\iota^G$. 

\medskip  We are now ready for the comparison of the eight inductions, which follows from   different  propositions in \cite{Abeparind}. In the  following propositions,   $\cV$ is any right $H(M)_R$-module.

\begin{maprop} \label{prop:exc} Exchanging  $+, - $ corresponds to the  twist    by $n_{w_Gw_M}$,
\begin{align} \label{eq:twist2'} 
\cV\otimes_{H(M^\epsilon),\theta^{\eta}}H(G) &\simeq    n_{w_Gw_M}(\cV)\otimes_{H(M^{op,-\epsilon}),\theta^{\eta}}H(G),\\ 
\label{eq:twist3'}
  \Hom_{H(M^\epsilon),\theta^{\eta}}(H(G), \cV)
&\simeq       \Hom_{H(M^{op,-\epsilon}),\theta^{\eta}}(H(G),n_{w_Gw_M}(\cV)).
 \end{align}
 Exchanging   $\theta, \theta^*$ corresponds to the  involutions $\iota^M_{\ell-\ell_M}$  and $ \iota^G$. 
   \begin{align}\label{eq:inv1}( \cV\otimes_{ H(M^\epsilon),\theta^{\eta}}H(G))^{\iota^G}&\simeq  \cV^{\iota^M_{\ell-\ell_M}}\otimes_{ H(M^{\epsilon}),\theta^{*\eta}}H(G) ,\\
 \label{eq:inv2e}
 \Hom_{H(M^\epsilon),\theta^{\eta}}(H(G), \cV)^{\iota^G} &\simeq \Hom_{H(M^\epsilon),\theta^{*\eta}}(H(G), \cV^{\iota^M_{\ell-\ell_M}}).   \end{align} 
Exchanging $\otimes ,\Hom$ corresponds to the  involutions $\iota^M_{\ell-\ell_M}$  and $\iota^G$,  \begin{align}\label{eq:inv3e} (\cV \otimes_{H(M^\epsilon),\theta^\eta}H(G))^{\iota^G} \simeq \Hom_{H(M^\epsilon),\theta^\eta}(H(G),  \cV^{\iota^M_{\ell-\ell_M}}). 
\end{align}
 \end{maprop}

\begin{remark}{\rm 
One can exchange  $(\otimes, \theta^\eta)$ and $(\Hom, \theta^{*\eta})$ without changing the isomorphism class:
\begin{equation}\label{eq:nothing}\cV\otimes_{ H(M^\epsilon),\theta^{\eta}}H(G) \simeq \Hom_{H(M^\epsilon),\theta^{*\eta}}(H(G), \cV).
\end{equation}}
\end{remark}

{\bf Duality} Let $\zeta$ the anti-involution of $H(G)$ defined by $\zeta(T_w)=T_{w^{-1}}$ for $w\in W$; we have also $\zeta(T^*_w)=T^*_{w^{-1}}$ \cite[Remark 2.12]{Vigadjoint}. 
 The dual  of a right $H(G)_R$-module $\cx$ is $\cx^*=\Hom_R(\cx,R)$ where $h\in H(G)_R$  acts on $f\in \pi^*$ by $(fh)(x)=f(x \zeta(h))$ \cite[Introduction]{Abeinv}.

\begin{maprop}\label{prop:duality} The dual exchanges  $(\otimes, + )$ and  $(\Hom, -)$:  
\begin{align}\label{eq:dual1}(\cV\otimes_{ H(M^\epsilon),\theta^{\eta}}H(G))^* \simeq \Hom_{H(M^{-\epsilon}),\theta^{\eta}}(H(G), \cV^*),\\
\label{eq:dual2} \cV^*\otimes_{ H(M^\epsilon),\theta^{\eta}}H(G)\simeq (\Hom_{H(M^{-\epsilon}),\theta^{\eta}}(H(G), \cV))^* . \end{align}
 \end{maprop}
\begin{proof} Applying  \eqref{eq:nothing}, the upper isomorphism  \eqref{eq:dual1}  for an arbitray $(\epsilon, \theta^\eta)$ is equivalent to the lower isomorphism  \eqref{eq:dual2}  for an arbitray $(\epsilon, \theta^\eta)$.  

We prove the upper isomorphism for  an arbitray $(\epsilon, \theta^\eta)$. For $ (+,  \theta)$, it is implicit in  \cite[\S 4.1]{Abeinv}. Applying it to the twist by $n_{w_Gw_P}$ of $(M,\cV)$ and using  \eqref{eq:twist2'} \eqref{eq:twist3'}, we get  \eqref{eq:dual1} for  $ (-,  \theta)$.  The image by $\iota^G$ of the upper isomorphism \eqref{eq:dual1} for $(\epsilon, \theta)$ and $\cV ^{\iota^M_{\ell-\ell_M}}$ is  \eqref{eq:dual1} for $(\epsilon, \theta^*)$ and  $\cV$, because   
the anti-involution $\zeta_M$  of $H(M)$ commutes with the involution $\iota^M_{\ell-\ell_M}$,  their composite in any order sends $(T_w^M, T_w^{M,*}) $ to $(-1)^{\ell(w)}(T_{w^{-1}}^{M,*},T_{w^{-1}}^M )$  for $w\in W_M$,   as $\ell(w)=\ell(w^{-1})$.   \end{proof}

 \end{document}